\documentclass[11pt]{article}
\usepackage[a4paper,tmargin=2.25cm,bmargin=2.25cm, rmargin=2.cm,lmargin=2.cm]{geometry}
\usepackage{amssymb, amsmath, mathtools, fancyhdr,amsfonts,empheq}
\usepackage{graphicx}
\usepackage[table]{xcolor}
\usepackage{booktabs, multirow}
\usepackage{stmaryrd}
\usepackage{amsthm, enumitem}
\usepackage{float}
\usepackage{caption}
\usepackage{subcaption}
\usepackage{chngcntr}
\usepackage{apptools}
\usepackage{scalerel}
\usepackage{undertilde}
\usepackage{array}
\AtAppendix{\counterwithin{lemma}{section}}
\definecolor{ppurple}{rgb}{.5,0,1}
\definecolor{lightblue}{rgb}{0.22,0.45,0.70}
\definecolor{lightgreen}{rgb}{0.22,0.50,0.25}
\definecolor{mangolassi}{RGB}{204,85,0}
\usepackage[colorlinks=true,breaklinks=true,linkcolor=lightblue,citecolor=lightgreen,urlcolor=lightblue]{hyperref}

\newcolumntype{g}{ >{\columncolor{mygrey}} c }
\newcolumntype{C}[1]{>{\centering\arraybackslash}m{#1}}
\newtheorem{lemma}{Lemma}[section]
\newtheorem{theorem}[lemma]{Theorem}

\newtheorem{remark}[lemma]{Remark}
\newtheorem{corollary}[lemma]{Corollary}

\numberwithin{equation}{section}
\numberwithin{figure}{section}
\numberwithin{table}{section}
\input mydef.sty

\date{\today}
\title{Improved $L^2$-error estimates for the wave equation discretized using hybrid nonconforming methods on simplicial meshes\thanks{This project has received funding from the European Union’s Horizon 2020 research and innovation programme under the Marie Skłodowska-Curie grant agreement No 101034255.}}
\author{
{\sc Bernardo Cockburn}\thanks{School of Mathematics, University of Minnesota, USA.
Email:{\tt bcockbur@umn.edu}}
\qquad
{\sc Alexandre Ern}\thanks{CERMICS, ENPC, Institut Polytechnique de Paris, F-77455 Marne-la-Vallée cedex 2,  and SERENA Project-Team, Centre Inria de Paris, F-75647 Paris, France. Email: {\tt alexandre.ern@enpc.fr}}
\qquad 
{\sc Rekha Khot}\thanks{SERENA Project-Team, Centre Inria de Paris, F-75647 Paris, and CERMICS, ENPC, Institut Polytechnique de Paris, F-77455 Marne-la-Vallée cedex 2,   France. Email: {\tt rekha.khot@inria.fr}}
}

\begin{document}
\maketitle
\begin{abstract}
We present improved $L^2$-error estimates on the time-integrated primal variable for the wave equation in its first-order formulation. The space discretization relies on a hybrid nonconforming method, such as the hybridizable discontinuous Galerkin, the hybrid high-order or the weak Galerkin methods. We consider both equal-order and mixed-order settings on simplices, and include the lowest-order case with piecewise constant unknowns on the faces and in the cells. Our main result is a superclose, resp., optimal bound on the above error in the equal-, resp., mixed-order case. A key result of independent interest to achieve these estimates are novel approximation estimates for an interpolation operator inspired from the hybridizable discontinuous Galerkin literature.
\end{abstract}



\section{Introduction}

The wave equation is a classical example of hyperbolic partial differential equation. 
We focus here on the first-order formulation in space and time, with a skew-symmetric differential operator in space, having the structure of a time-dependent Friedrichs's system. The first-order formulation is attractive since it makes the application of high-order time discretization schemes such as Runge--Kutta (RK) schemes rather straightforward. It is also more natural than the second-order formulation in time in the context of more complex models derived from conservation laws. The space discretization of Friedrichs's systems is usually done using discontinuous Galerkin (dG) methods or stabilized $H^1$-conforming finite elements (see, e.g., \cite{JohPi:86,ErnGu:06} and \cite[Chap.~58--60]{ErnGuermondIII} for an overview). In the time-dependent case, the combination of RK and dG methods has become a popular paradigm, in the wake of \cite{cockburn1989tvb}. See also \cite[Chapter~77-78]{ErnGuermondIII} for further insight.

Hybrid nonconforming discretization schemes in space are based on approximating the unknowns inside the mesh cells as well as their trace on the mesh faces. The main advantage of such methods with respect to dG is their reduced computational cost owing to static condensation, which eliminates locally all the cell degrees of freedom. The most popular hybrid nonconforming method is the hybridized discontinuous Galerkin (HDG) method introduced in \cite{cockburn2009}. A different devising viewpoint was developed some years later, leading to the hybrid high-order (HHO) and weak Galerkin (WG) methods. Both methods were developed independently, the former was coined in \cite{DE_2015, DEL_2014} and the latter in \cite{Wang2013}. In a nutshell, HDG methods approximate a triple composed of the solution, its gradient and its traces on the mesh skeleton and the key devising concept is the numerical flux trace, whereas HHO and WG approximate a pair composed of the solution and its traces on the mesh skeleton and the key devising concepts are local gradient reconstruction and local stabilization operators (the gradient reconstruction operator being called weak gradient in WG). We refer the reader to \cite{CDE_2016,cicuttin2021hybrid} and \cite[Subsection 6.6]{CockburnDurham16} for more details on the connections between HDG/HHO/WG methods. 
Hybrid nonconforming methods have already been employed to approximate the wave equation in its first-order form. The HDG  method has  been studied in \cite{cockburn2014uniform,stanglmeier2016explicit,cockburn2018stormer} (see also formulations exploiting the Hamiltonian structure of the equations to combine HDG space discretizations with symplectic time-marching methods \cite{SanchezCiucaNguyenPeraireCockburn17,SanchezCockburnNguyenPeraire21,SanchezDuCockburnNguyenPeraire2022,cockburnDuSanchez22}), the HHO method in \cite{BDE_2022, BDES_2021} (see also \cite{ern2024explicit}),
and the WG method in \cite{zhang2022weak,zhang2022explicit}.

The main contribution of the paper is a superconvergent $L^2$-error estimate on the time-integrated primal variable. The presentation uses the HHO/WG formalism based on a pair of discrete unknowns, but
we briefly indicate the connection to the HDG formalism based on a triple of discrete unknowns. 
Our analysis covers both equal-order and mixed-order cases, whereby the polynomial degree of the cell unknowns is the same or one order higher, respectively, than the degree of the face unknowns, 
$k\ge0$. In both cases, the polynomial degree of the gradient reconstruction is the same as the degree of the face unknowns. The stabilization is scaled by the reciprocal of the mesh size, and employs the high-order HHO correction in the equal-order case \cite{DEL_2014}, whereas it coincides with the Lehrenfeld--Sch\"oberl HDG stabilization in the mixed-order case \cite{lehsc:16}. In all cases, provided the exact solution has optimal smoothness, we prove that the decay rate on the $L^2$-error on the time-integrated primal variable is $\mathcal{O}(h^{k+1+s})$, where $h$ is the mesh size and $s\in(\frac12,1]$ is the index of elliptic regularity pickup. 

To appreciate the improvement brought by our analysis, we summarize in Table~\ref{tab} the decay rates available on the error. The energy error, that is, the $L^2$-error on both primal and dual variables, decays at rate $\mathcal{O}(h^{k+1})$ \cite{BDES_2021}. For the dual variable, this rate is indeed optimal, whereas, for the primal variable, this rate is optimal in the equal-order case and suboptimal in the mixed-order case. Here, we improve the estimate on the primal variable by considering its time integration. Indeed, we show that for this time-integrated variable, the decay rate is $\mathcal{O}(h^{k+1+s})$, which means a supercloseness estimate in the equal-order case and an optimal estimate in the mixed-order case.
The idea of considering a time-integrated variable comes from
\cite{cockburn2014uniform}, where an improved $L^2$-error estimate is shown using equal-order HDG and plain Least-Squares stabilization. The main differences with the present work are that \cite{cockburn2014uniform} assumes full elliptic regularity and $k\geq 1$, and requires a local post-processing using the discrete dual variable. Here, we consider more general stabilization strategies, any elliptic regularity pickup $s\in(\frac12,1]$, include the case $k=0$, and do not require using the discrete dual variable.

\begin{table}[H]
	\begin{center}
		\begin{tabular}{|C{2cm}|C{2cm}|C{2cm}|C{2cm}|C{2cm}|C{2.5cm}|}		
			\hline
		&
			\multicolumn{2}{c|}{\multirow{2}{*}{Cell polynomial degree}}	 &\multicolumn{3}{c|}{Convergence rate} \\
			\cline{4-6} \multirow{2}{*}{Setting}\rule{0pt}{1em} 	& \multicolumn{2}{c|}{}  & \multicolumn{2}{c|}{Energy error} & $L^2$-error\\
			\cline{2-6}  \rule{0pt}{1.5em} & Primal ($v$) & Dual ($\bsg$) & Primal ($v$) & Dual ($\bsg$) & Primal ($\int_0^t v$)\\
			\hline
			\multirow{2}{*}{Equal-order} & \multirow{2}{*}{$k$} & \multirow{2}{*}{$k$} & \multirow{2}{*}{$\boldsymbol{k+1}$} & \multirow{2}{*}{$k+1$} & \multirow{2}{*}{$\boldsymbol{k+1+s}$} \\
			& & & & & \\
			\hline
			\multirow{2}{*}{Mixed-order} & \multirow{2}{*}{$k+1$} & \multirow{2}{*}{$k$} & \multirow{2}{*}{$\boldsymbol{k+1}$} & \multirow{2}{*}{$k+1$} & \multirow{2}{*}{$\boldsymbol{k+1+s}$} \\
			& & & & & \\
			\hline 
		\end{tabular}
		\caption{Convergence rates of energy and $L^2$-errors in  equal- and mixed-order settings given a polynomial degree $k\geq 0$. In all cases, the face polynomial degree for the primal variable is $k$.}
		\label{tab}
	\end{center}
\end{table}
Our convergence analysis hinges on an auxiliary result of independent and broader interest, namely an interpolation operator with supercloseness (equal-order) or optimal (mixed-order) approximation properties.  
This operator in the mixed-order case is similar to the so-called HDG+-interpolation operator \cite{du2019invitation} (see also the HDG-interpolation operator from \cite{cockburn2010projection}). In the equal-order case, this operator is, to our knowledge, novel. Moreover, some of the arguments in the proof of its approximation estimates are slightly simpler than those in \cite{cockburn2010projection}. Another relevant improvement in the analysis is that we relax the regularity requirement on the divergence of the dual variable, which is important if one wants to address the case of partial elliptic regularity pickup (less than 1). We call our novel operator the H-interpolation operator, and hope that it will find further applications in the analysis of HDG/HHO/WG methods. In the present setting, the main role of the H-interpolation operator lies in the handling of the initial conditions on the time derivatives. 

The paper is organized as follows: Section~\ref{sec:model-disc} introduces the model problem and the HDG/HHO/WG discretization. Section~\ref{sec:main} collects our main results on the H-interpolation operator and on the error estimates for the wave equation. Section~\ref{sec:aux} is devoted to all preliminary results needed in Section~\ref{sec:H-Int} to establish the well-posedness of the H-interpolation operator and the  approximation estimates. Finally, Section~\ref{sec:errest} contains the error analysis, first in the energy norm and then in the $L^2$-norm for the time-integrated primal variable.
For illustrative numerical experiments related to our theoretical results, we refer the reader to \cite{BDE_2022, mottier2025hybrid}.

\section{Model problem and space semi-discretization}
\label{sec:model-disc}

This section presents the model problem and its space semi-discretization using hybrid nonconforming methods.

\subsection{Model problem}\label{sec:model}
Let $\Omega$ be a polyhedral Lipschitz domain (open bounded connected set)  in $\mathbb{R}^d$ for $d\in\{2,3\}$ with boundary $\Gamma$. The first-order formulation of the acoustic wave equation defined on the space domain $\Omega$ and the time domain $J:=(0,T_f)$ for the final time $T_f>0$ consists of the coupled partial differential equations in $J\times\Omega$,
\begin{subequations}\label{model}
\begin{align}
\pt\bsg - \nabla v &= \bzero,\label{eq:model.1a}\\
\pt v-\nabla{{\cdot}}\bsg &= f,\label{eq:model.1b}
\end{align}
\end{subequations}
involving as unknowns the dual variable $\bsg:J\times\Omega\to\rbb^d$ and the primal variable $v:J\times\Omega\to\rbb$, and the known source term $f:J\times\Omega\to\rbb$. The initial conditions are
\begin{align*}
\bsg(0) = \bsg_0,\quad v(0) = v_0 \quad\text{in}\;\Omega,
\end{align*}
with given data $\bsg_0:\Omega\to\rbb^d$ and $v_0:\Omega\to\rbb$,	and the boundary condition is (for simplicity) 
\begin{align*}
v = 0\quad\text{on}\;J\times\Gamma.
\end{align*}
We follow the standard notation for Sobolev/Lebesgue spaces  and respective norms. The $L^2$-inner product and the associated norm on a domain $S$ are denoted by $(\cdot,\cdot)_S$ and $\|\cdot\|_S$ respectively. We use a bold symbol for vector variables and spaces composed of vector-valued fields, e.g., $\bL^2(\Omega) := [L^2(\Omega)]^d$.

We set $\ul{V}:=\bH(
\dv,\Omega)\times H^1_0(\Omega)\subset \ul{L}:=\bL^2(\Omega)\times L^2(\Omega)$. Here and in what follows, we underline symbols that refer to a pair composed of one dual variable and one primal variable. Assuming an initial condition $(\bsg_0,v_0)\in \ul{V}$, the Hille--Yosida theorem gives a solution $(\bsg,v)\in C^0(\ol{J};\ul{V})\cap C^1(\ol{J};\ul{L})$. In particular, we have, for all $t\in \ol{J}$, 
\begin{subequations}\label{weak-form}
\begin{alignat}{2}
(\pt\bsg(t),\bxi)_{\Omega} - (\nabla v(t), \bxi)_{\Omega}&=0,\qquad&\forall& \bxi\in \bL^2(\Omega), \\
(\pt v(t),w)_{\Omega} + (\bsg(t), \nabla w)_{\Omega} &= (f(t),w)_{\Omega}\qquad&\forall& w\in H^1_0(\Omega).
\end{alignat}
\end{subequations}
Below, we use this weak formulation to define our space semi-discretization schemes.

\subsection{Space semi-discretization}\label{sec:space_semidisc}	

This section presents the HHO/WG space semi-discretization of the model problem and its equivalent  HDG rewriting.

\subsubsection{Mesh and polynomial spaces}
Let $\cT$ be a simplicial mesh covering exactly the domain $\Omega$. The set  $\cF$ contains the mesh faces, and is divided into the set of mesh interfaces $\cF^\circ$ and the set of mesh boundary faces $\cF^\partial$. A generic mesh cell is denoted by $T\in\cT$ with diameter $h_T$, unit outward normal $\bn_T$ on $\dT$, and the set $\cF_{\dT}$ collects the mesh faces located on the boundary of $T$.

Let $k\ge0$ be the polynomial degree and let $k'\in\{k,k+1\}$.  We set 
\begin{equation*}
\bS_\cT^k:=\bigtimes_{T\in\cT}\bS_T^k,
\end{equation*}
with $\bS_T^k:=\pbb^k(T;\rbb^d)$, as well as 
\begin{equation*}
\h{V}_\cM^{k}:=V_\cT^{k'}\times V_\cF^k,\qquad V_\cT^{k'}:=\bigtimes_{T\in\cT}V_T^{k'},\qquad V_\cF^{k}:=\bigtimes_{F\in\cF}V_F^{k},
\end{equation*}
where $V_T^{k'}:=\pbb^{k'}(T;\rbb)$ (resp. $V_F^{k}:=\pbb^k(F;\rbb)$) is composed of the scalar-valued $d$-variate (resp. $(d-1)$-variate) polynomials of degree at most $k'$ (resp. $k$) restricted to the cell $T$ (resp. face $F$).
Here, we use the subscript $\cM:=(\cT,\cF)$ to indicate the joint collection of mesh cells and faces. For a generic $\bxi_\cT\in \bS_\cT^k$ and a generic $\h{w}_\cM \in \h{V}_\cM^{k}$, we write
\begin{equation*}
\bxi_\cT := (\bxi_T)_{T\in\cT}, \qquad
\h{w}_\cM := (w_\cT,w_\cF) := \big( (w_T)_{T\in\cT}, (w_F)_{F\in\cF} \big),
\qquad
w_{\dT} := (w_F)_{F\in \cF_{\dT}}.
\end{equation*}
In what follows, we use a hat symbol, as in $\h{v}$, to refer to a pair composed of a cell unknown $v_T$, and a face unknown, $v_{\partial T}$. To impose the zero Dirichlet boundary condition on the primal variable, we define $\h{V}_{\cM0}^{k}:=V_\cT^{k'}\times V_{\cF0}^k$ with $V_{\cF0}^k:=\{v_\cF\in V_\cF^k: v_F =0 \quad\forall F\in \cF^\partial\}$. Let $\Pi^{k'}_T$ (resp. $\Pi^k_{F}$) be the $L^2(T)$-orthogonal (resp. $L^2(F)$-orthogonal) projection onto $V_T^{k'}$ (resp. $V^k_{F}$). Let $\Pi^{k'}_\cT$ (resp. $\Pi^k_{\cF}$) be the piecewise $L^2$-orthogonal projection onto $V_\cT^{k'}$ (resp. $V^k_{\cF}$). Recall the following approximation estimate: For all $T\in\cT$ and all $v\in H^{m}(T)$ with $m\in\{0{:}k'\}$, 
\begin{align}
\norm{v-\Pi^{k'}_T(v)}_{H^r(T)} \lesssim h_T^{m-r}|v|_{H^m(T)}\qquad\forall r\in\{0{:}m\}.\label{L2:proj-est}
\end{align}
Similar definitions and estimates hold for vector-valued functions. We also define the elliptic projection ${\cal{E}}_T^{k+1}:H^1(T)\to \pbb^{k+1}(T;\rbb)$ as 
\begin{subequations}
	\begin{align}
		(\nabla({\cal{E}}_T^{k+1}(v)),\nabla q)_T &= (\nabla v,\nabla q)_T\qquad\forall q\in \pbb^{k+1}_*(T;\rbb):=\pbb^{k+1}(T;\rbb)/\rbb,\\
		({\cal{E}}_T^{k+1}(v),1)_T &= (v,1)_T,
	\end{align}
\end{subequations} 
which satisfies the following approximation estimate:
\begin{align}
\norm{\nabla(v-{\cal{E}}_T^{k+1}(v))}_{T} \lesssim h_T^{k+1}|v|_{H^{k+1}(T)}.\label{est:elliptic}
\end{align}

\subsubsection{HHO/WG space semi-discretization}
We approximate  the primal variable $v$ with an HHO/WG method using cell polynomials of degree $k'$ and face polynomials of degree $k$. The setting is said to be of equal-order if $k'=k$ and of mixed-order if $k'=k+1$. On the other hand, we approximate the dual variable $\bsg$ with a dG method using cell polynomials of degree $k$. The stabilization is the one introduced in HHO methods for the equal-order case and in HDG methods for the mixed-order case, whereas WG methods in the equal-order case generally employ a plain Least-Squares stabilization (see Remark~\ref{rem:stab}).

The space semi-discretized wave equation reads as follows: Find $\bsg_\cT:\ol{J}\rightarrow \bS_\cT^k$ and $\h{v}_\cM:\ol{J}\rightarrow \h{V}_{\cM0}^{k}$ such that, for all $t\in \ol J$,  all $\bxi_\cT\in\bS_\cT^k$ and all $\h{w}_\cM\in \h{V}_{\cM0}^{k}$,
\begin{subequations} \label{eq:HHO} 
\begin{align}
(\pt \bsg_\cT(t),\bxi_\cT)_{\Omega} - (\bG_\cT(\h{v}_\cM(t)),\bxi_\cT)_\Omega &= 0, 
\label{eq:HHO_flux} \\
(\pt v_\cT(t),w_\cT)_{\Omega} + (\bsg_\cT(t),\bG_\cT(\h{w}_\cM))_\Omega 
+ s_\cM(\h{v}_\cM(t),\h{w}_\cM) &= (f,w_\cT)_\Omega,
\label{eq:HHO_primal}
\end{align} 
\end{subequations}
where the gradient reconstruction operator $\bG_\cT : \h{V}_\cM^{k} \rightarrow \bS_T^k$ is such that, for all $\h{v}_\cM \in \h{V}_\cM^{k}$ and all $\bxi_\cT \in \bS_T^k$,
\begin{align}
(\bG_\cT(\h{v}_\cM),\bxi_\cT)_\Omega := {}& \sum_{T\in\cT} \big\{ 
- (v_T,\dv \bxi_T)_T + (v_{\dT},\bxi\SCAL\bn_T)_{\dT} \big\} \nonumber \\
= {}& \sum_{T\in\cT} \big\{ 
(\nabla v_T,\bxi_T)_T - (v_T-v_{\dT},\bxi\SCAL\bn_T)_{\dT} \big\}.\label{def:G}
\end{align}
The stabilization bilinear form $s_\cM:\h{V}_{\cM}^{k}\times \h{V}_{\cM}^{k} \rightarrow \rbb$ is such that
\begin{equation}
\label{def:tau}
s_{\cM}(\h{v}_\cM,\h{w}_\cM) := \sum_{T\in\cT} \tau_{T} \big(S_{\dT}^{\sPi}(v_T,v_{\dT}),S_{\dT}^{\sPi}(w_T,w_{\dT})\big)_{\dT},\quad \text{with}\quad \tau_T := \ell_\Omega h_T^{-1},
\end{equation}
where the scaling factor $\ell_\Omega:=\text{diam}(\Omega)$ is introduced to make
the weight $\tau_T$ non-dimensional,
and the operator $S_{\dT}^{\sPi}$ is defined as
\begin{subequations}\label{HHO:stab}
\begin{empheq}[left={S_{\dt}^{\HHO}(\h{v}_T):=\empheqlbrace}]{alignat = 2}
& S_{\dt}^{\rm{eo}}(\h{v}_T):=\Pi^k_{\dt}
\{\delta_{\dt}(\h{v}_T)+((1-\Pi^k_T)R_T(\h{v}_T))|_{\dt}\}\qquad&\text{for}\; &k'=k,\label{HHO-stab:1a}\\
&S_{\dt}^{\rm{mo}}(\h{v}_T):=\Pi^k_{\dt}\{\delta_{\dt}(\h{v}_T)\}&\text{for}\;& k'=k+1,\label{HHO-stab:1b}
\end{empheq}
\end{subequations}
with the boundary difference operator 
\begin{align}
\delta_{\dt}(\h{v}_T):= v_T|_{\dt}-v_{\dt}.\label{diff-op}
\end{align}
Here, $\Pi^k_{\dT}$ is the $L^2$-orthogonal projection onto the space $\pbb^k(\cF_{\dT};\rbb) :=\bigtimes_{F\in\cF_T}V_F^{k}$ so that  $S_{\dT}^{\sPi}(\h{v}_T) \in \pbb^k(\cF_{\dT};\rbb)$ in both cases.  Moreover, the local potential reconstruction operator $R_T:\h{V}_T^k\to \pbb^{k+1}(T;\rbb)$ is defined, for all $\h{v}_T\in\h{V}_T^k$, such that
\begin{subequations}\label{def:reconst}
\begin{align}
(\nabla R_T(\h{v}_T),\nabla w)_{T} &= (\nabla v_T, \nabla w)_{T}-(v_T-v_{\dt},\nabla w{\cdot}\bn_T)_{\dt}\qquad\forall w\in\pbb^{k+1}_{\mbox{*}}(T;\rbb),\label{def:reconst.a}\\
(R_T(\h{v}_T)-v_T,1)_{T}&=0.\label{def:reconst.b}
\end{align}
\end{subequations} 
In \eqref{HHO-stab:1a}, we notice that the correction term involving the operator $R_T$ is added to the plain least-squares stabilization classically considered in the context of dG methods. 
Finally,  \eqref{eq:HHO} is initialized by prescribing
\begin{equation} \label{eq:IC}
\bsg_\cT(0) := \bPi^{\bsg}_\cT(\bsg_0,v_0),\qquad
v_\cT(0) := \Pi^v_\cT(\bsg_0,v_0),
\end{equation}
where the projections $(\bPi^{\bsg}_\cT,\Pi^v_\cT)$ are yet to be defined. 
\begin{remark}[Role of stabilization] \label{rem:stab}
Taking the stabilization operator equal to the difference operator $\delta_{\dt}$ corresponds to plain Least-Squares stabilization. Here, we consider more sophisticated choices that play a key role in achieving higher-order estimates. In particular, the stabilization operator enters the definition of the H-interpolation operator (see~\eqref{H:proj.c} below). This choice, in turn, is instrumental in achieving improved approximation estimates on the primal variable (see \eqref{est:pw} and Remark~\ref{rem:sup} below).
\end{remark}
\subsubsection{HDG rewriting}
The discrete problem \eqref{eq:HHO} can be recast in the HDG setting, see \cite[Section~4.3]{BDE_2022} and \cite{CDE_2016} for details. Here,  we recall the local discrete problem which is obtained by  taking test functions with support localized in the mesh cell $T\in\cT$, using the definition of $\bG_\cT$ in \eqref{eq:HHO_flux}-\eqref{eq:HHO_primal}, and  rewriting the stabilization term in \eqref{eq:HHO_primal}. Notice from the definition of the operator $R_T$ that 
\begin{align}
R_T(\h{v}_T) = R_T(v_T,v_T|_{\dt})-R_T(0,\delta_{\dt}(\h{v}_T))=v_T-R_T(0,\delta_{\dt}(\h{v}_T)).
\end{align} 
Hence, we can rewrite $S_{\dt}^{\rm{eo}}(\h{v}_T) = \Pi^k_{\dt}\{\delta_{\dt}(\h{v}_T)-(1-\Pi^k_T) R_T(0,\delta_{\dt}(\h{v}_T))|_{\dt}\}$. This together with the definition of $S_{\dt}^{\rm{mo}}$ shows that in both cases $S_{\dt}^{\HHO}$ only  acts on $\delta_{\dt}(\h{v}_T)$. Thus, we define the operator $\ti{S}_{\dt}^{\HHO}:\pbb^k(\cF_{\dt};\rbb)\to \pbb^k(\cF_{\dt};\rbb)$ by
\begin{subequations}\label{HHO:stab-re}
\begin{empheq}[left={\ti{S}_{\dt}^{\HHO}(\mu):=\empheqlbrace}]{alignat = 2}
& \ti{S}_{\dt}^{\rm{eo}}(\mu):=\Pi^k_{\dt}
\{\mu-((1-\Pi^k_T)R_T(0,\mu))|_{\dt}\}\qquad&\text{for}\; &k'=k,\label{HHO-stab-re:1a}\\
&\ti{S}_{\dt}^{\rm{mo}}(\mu):=\Pi^k_{\dt}\{\mu\}&\text{for}\;& k'=k+1.\label{HHO-stab-re:1b}
\end{empheq}
\end{subequations}
For all $T\in\cT$ and all $(\bxi_T,w_T)\in\bS_T^k\times V_T^{k'}$, we infer from \eqref{eq:HHO} that
\begin{subequations}
\begin{align}
(\pt \bsg_T(t),\bxi_T)_{T} +(v_T(t), \nabla{\cdot}\bxi_T)_T - (v_{\dt}(t),\bxi_T{\cdot}\bn_T)_{\dt} &= 0,\\
(\pt v_T(t),w_T)_{T}+(\bsg_T(t),\nabla w_T)_T -  (\h{\bsg}_{\dt}(t),w_T)_{\dt}&=(f(t),w_T)_T.
\end{align}
\end{subequations}
For all $F\in\cF^\circ$ with $F\in \partial T^-\cap \partial T^+$  and all $w_F\in V_{F}^k$, we infer from \eqref{eq:HHO_primal} that
\begin{align}\tag{3.13c}
(\h{\bsg}_{\partial T^-}(t)+\h{\bsg}_{\partial T^+}(t),w_{F})_{F}=0,
\end{align}
with the numerical flux trace
\begin{empheq}[left={\h{\bsg}_{\dt}(t):=\bsg_T(t){\cdot}\bn_T-\tau_T\lambda_{\dt}(\delta_{\dt}(\h{v}_T(t)))\quad\text{with}\quad \lambda_{\dt}(\cdot):=\empheqlbrace}]{alignat* = 2}
& ((\ti{S}_{\dt}^{\rm{eo}})^*\circ \ti{S}_{\dt}^{\rm{eo}})(\cdot)\qquad&\text{for}\; &k'=k,\\
&\Pi^k_{\dt}(\cdot)&\text{for}\;& k'=k+1,
\end{empheq}
where $(\ti{S}_{\dt}^{\rm{eo}})^*:\pbb^k(\cF_{\dt};\rbb)\to \pbb^k(\cF_{\dt};\rbb)$ denotes the adjoint operator of $\ti{S}_{\dt}^{\rm{eo}}$ with the respect to $L^2(\dt)$-inner product.

\begin{remark}[Initial face values]
We notice that the value $v_\cF(0)$ is not prescribed in \eqref{eq:IC}, but results from 
\eqref{eq:HHO_primal} by testing with an arbitrary function $\h{w}_\cM=(0,w_\cF)$. Let $(\bsg_\cT,\h{v}_{\cM})$ solve \eqref{eq:HHO}. A simple calculation shows that the following holds: For all $F\in \cF^\circ$, 
\begin{subequations}
\begin{empheq}[left={v_F(0)=\empheqlbrace}]{alignat = 2}
&\frac{1}{\avg{\tau_\cT}_F} \;\;\;\;\;\;\avg{\tau_\cT v_\cT(0)}_F\hskip.2truecm+\frac{1}{2\avg{\tau_\cT}_F} \lambda_{\dt}^{-1}(\jump{\bsg_\cT(0)}_{\dt}\SCAL\bn_T)|_F ,\qquad&\text{for}\; &k'=k,\label{3.12a}\\
& \frac{1}{\avg{\tau_\cT}_F} \Pi^k_F\big(\avg{\tau_\cT v_\cT(0)}_F\big)+\frac{1}{2\avg{\tau_\cT}_F} \;\;\;\;\;\;\;\jump{\bsg_\cT(0)}_F\SCAL\bn_F,&\text{for}\;& k'=k+1,
\end{empheq}
\end{subequations}
and $v_F(0)=0$ for all $F\in\cF^\partial$. Here, for all $F\in\cF^\circ$, letting $T^\pm$ be the two mesh cells sharing $F$, and defining $\bn_F$ as the unit normal to $F$ pointing from $T^-$ to $T^+$ (the orientation is arbitrary but fixed once and for all), the average and jump of a piecewise smooth function at $F$, say $\phi$, are defined as
\begin{equation*}
\avg{\phi}_F := \frac12 \big( \phi|_{T^-}|_F+\phi|_{T^+}|_F\big),
\qquad
\jump{\phi}_F := \phi|_{T^-}|_F - \phi|_{T^+}|_F,
\end{equation*}
with  $\jump{\phi}_{\dt}|_F:=\jump{\phi}_F$ for all $F\in\cF_{\dt}$.
Notice that $\lambda_{\dt}$ in the equal-order case \eqref{3.12a} 
can be viewed as a term enforcing a stabilization mechanism. Indeed, for all $\mu\in\pbb^k(\cF_{\dt};\rbb)$, we have 

\[(\mu,{\lambda_{\dt}(\mu)})_{\dt}=(\mu,{((\ti{S}_{\dt}^{\rm{eo}})^*\circ \ti{S}_{\dt}^{\rm{eo}})(\mu)})_{\dt} = \norm{\ti{S}_{\dt}^{\rm{eo}}(\mu)}^2_{\dt}=\norm{S_{\dt}^{\rm{eo}}(0,\mu)}^2_{\dt}\gtrsim \norm{\mu}^2_{\dt}, \]
where the lower bound follows from the stability estimate of Lemma~\ref{prop:stability} below with $\h{v}_T:=(0,\mu)$. In particular, the above bound shows that $\lambda_{\dt}^{-1}$ is well-defined.
\end{remark}

\subsection{The H-interpolation operator}\label{sec:H-interp}
In this section, we introduce our novel, key tool for the error analysis, namely, the H-interpolation operator. We consider  both equal- and mixed-order settings. 
We first provide a unified definition of the H-interpolation operator, and then discuss separately the equal- and mixed-order cases.

\noindent\textbf{Unified H-interpolation operator}. Let $T\in\cT$. For all $\ul{v}:=(\bsg,v)\in \bH^{\nu_{\bsg}}(T)\times H^{\nu_v}(T)$ with $\nu_{\bsg}>\frac{1}{2}$ and  $\nu_v>\frac{1}{2}$, the H-interpolation operator 
\begin{align*}
\ul{\Pi}^{\sH}_T(\ul{v}) :=(\bPi^{\bsg}_T(\ul{v}),\Pi^v_T(\ul{v})) \in \pbb^k(T;\rbb^d) \times \pbb^{k'}(T;\rbb)
\end{align*}
on any simplex $T\in\cT$ is defined by the following equations:
\begin{subequations} \label{H:proj} \begin{alignat}{2}
&(\Pi^v_T(\ul{v})-v,w)_T = 0 &\qquad &\forall w\in \pbb^{k-1}(T;\rbb),\label{H:proj.a}\\
&(\bPi^{\bsg}_T(\ul{v})-\bsg,\bxi)_T = 0 &\qquad &\forall \bxi\in \pbb^{k-1}(T;\rbb^d),\label{H:proj.b} \\
&((\bsg-\bPi^{\bsg}_T(\ul{v}))\SCAL\bn_T,\mu)_{\dt} = \tau_T (S_{\dT}^{\HHO}(\Pi^v_T(\ul{v}),\Pi^k_{\dT}(v|_{\dt})),S_{\dt}^{\HHO}(0,\mu))_{\dT}  &\qquad&\forall \mu \in \pbb^k(\cF_{\dT};\rbb),\label{H:proj.c}\\
&(\bPi^{\bsg}_T(\ul{v})-\bsg,\nabla q)_T = ((\bPi^k_{\dT}(\bsg|_{\dt})-\bsg)\SCAL\bn_T,q)_{\dT} &\qquad &\forall q \in \pbb^{k'}(T;\rbb), \label{H:proj.d}
\end{alignat} 
\end{subequations}
recalling the convention that $\pbb^{-1}(T;\rbb) = \{0\}$ and $\pbb^{-1}(T;\rbb^d) = \{\bold{0}\}$. Recall also that  $S_{\dt}^{\HHO} := S_{\dt}^{\rm{eo}}$ for $k'=k$ and $S_{\dt}^{\HHO} := S_{\dt}^{\rm{mo}}$ for $k'=k+1$  in \eqref{H:proj.c}. Moreover, since \eqref{H:proj.b} and the $L^2$-orthogonality of $\Pi^k_{\dt}$ imply that
\begin{align*}
(\bPi^{\bsg}_T(\ul{v})-\bsg,\nabla q)_T = 0 =  ((\bPi^k_{\dT}(\bsg|_{\dt})-\bsg)\SCAL\bn_T,q)_{\dT}\qquad\forall q\in\pbb^k(T;\rbb),
\end{align*} 
the equation  \eqref{H:proj.d} is redundant for $k'=k$, and
reduces to taking $q$ in the linear space $\tilde{\pbb}^{k+1}(T;\rbb)$ composed of all homogeneous polynomials of degree ($k+1$) for $k'=k+1$.  Altogether, we can rewrite the explicit definitions in the equal- and mixed-order cases as follows:

\noindent\textbf{Equal-order case}. For $k'=k$, the H-interpolation operator from \eqref{H:proj} is defined on any simplex $T\in\cT$ as follows:
\begin{subequations}\label{H:int}
\begin{alignat}{2}
(\pw_T(\uv)-v, w)_{T} &= 0 &\forall& w\in\pbb^{k-1}(T;\mathbb{R}),\label{H:1a}\\
(\ps_T(\uv)-\bsg,\bxi)_{T} &= 0  &\forall& \bxi\in\pbb^{k-1}(T;\mathbb{R}^d),\label{H:1b}\\
((\bsg-\ps_T(\uv)){\cdot}\bn_T,\mu)_{\partial T} &= \tau_T(S_{\dt}^{\rm{eo}}(\pw_T(\uv),\Pi^k_{\dt}(v|_{\dt})),S_{\dt}^{\rm{eo}}(0,\mu))_{\dt}\quad &\forall& \mu\in\pbb^{k}(\cF_{\partial T};\rbb).\label{H:1c}
\end{alignat}
\end{subequations}

\noindent\textbf{Mixed-order case}. For $k'=k+1$, the H-interpolation operator from \eqref{H:proj} is defined on any simplex $T\in\cT$ as follows:
\begin{subequations} \label{eq:def2_proj} \begin{alignat}{2}
&(\Pi^v_T(\ul{v})-v,w)_T = 0 &\quad &\forall w\in \pbb^{k-1}(T;\rbb),\label{eq:def2_proj.c}\\
&(\bPi^{\bsg}_T(\ul{v})-\bsg,\bxi)_T = 0 &\quad &\forall \bxi\in \pbb^{k-1}(T;\rbb^d),\label{eq:def2_proj.a} \\
&((\bsg-\bPi^{\bsg}_T(\ul{v}))\SCAL\bn_T,\mu)_{\dt} = \tau_T (S_{\dT}^{\rm{mo}}(\Pi^v_T(\ul{v}),\Pi^k_{\dT}(v|_{\dt})),S_{\dt}^{\rm{mo}}(0,\mu))_{\dT}  &\quad&\forall \mu \in \pbb^k(\cF_{\dT};\rbb),\label{eq:def2_proj.d}\\
&(\bPi^{\bsg}_T(\ul{v})-\bsg,\nabla \tilde{q})_T = ((\bPi^k_{\dT}(\bsg|_{\dt})-\bsg)\SCAL\bn_T,\tilde{q})_{\dT} &\quad &\forall \tilde{q} \in \tilde{\pbb}^{k+1}(T;\rbb). \label{eq:def2_proj.b}
\end{alignat} \end{subequations} 
The first three equations \eqref{eq:def2_proj.c}-\eqref{eq:def2_proj.d} are formally the same as in the equal-order case, see \eqref{H:1a}-\eqref{H:1c} (but the definition of the stabilization operator changes) and one more equation, \eqref{eq:def2_proj.b}, is added because of the additional degrees of freedom in the mixed-order case. 

\begin{remark}[Comparison with \cite{du2019invitation}]
	Let $\ul{v} :=(\bsg,v)\in (\bH^{\nu_{\bsg}}(T)\cap \bH(\nabla\cdot;T))\times H^{\nu_v}(T)$ with $\nu_{\bsg}>\frac{1}{2}$ and $\nu_v>\frac{1}{2}$. In the mixed-order case,  the action of the $\rm{H}$-interpolation operator gives the same result as the operator $(\ti{\bPi}{}^{\bsg}_T,\ti{\Pi}^v_T)$, defined in \cite[Chapter~4]{du2019invitation} as follows:
\begin{subequations}\label{hdg}
\begin{align}
&(\ti{\Pi}^v_T(\ul{v})-v,w)_T = 0 &\qquad &\forall w\in \pbb^{k-1}(T;\rbb),\label{hdg:a}\\
&(\ti{\bPi}{}^{\bsg}_T(\ul{v})-\bsg,\bxi)_T = 0 &\qquad &\forall \bxi\in \pbb^{k-1}(T;\rbb^d),\label{hdg:b} \\
&((\bsg-\ti{\bPi}{}^{\bsg}_T(\ul{v}))\SCAL\bn_T,\mu)_{\dt} =  \tau_T (S_{\dT}^{\rm{mo}}(\ti{\Pi}^v_T(\ul{v}),\Pi^k_{\dT}(v|_{\dt})),S_{\dt}^{\rm{mo}}(0,\mu))_{\dT} &\qquad&\forall \mu \in \pbb^k(\cF_{\dT};\rbb),\label{hdg:c}\\
&(\nabla\SCAL(\ti{\bPi}{}^{\bsg}_T(\ul{v})-\bsg), \tilde{q})_T = \tau_T (S_{\dt}^{\rm{mo}}(\ti{\Pi}^v_T(\ul{v}),\Pi^k_{\dT}(v|_{\dt})), S_{\dt}^{\rm{mo}}(\tilde{q},0))_{\dt} &\qquad &\forall \tilde{q} \in \tilde{\pbb}^{k+1}(T;\rbb).\label{hdg:d}
\end{align}
\end{subequations}
Indeed,	for all $\tilde{q}\in\tilde{\pbb}^{k+1}(T;\rbb)$, using \eqref{hdg:d} and  $S_{\dt}^{\rm{mo}}(\tilde{q},0)= -S_{\dt}^{\rm{mo}}(0,\Pi^k_{\dt}(\tilde{q}|_{\dt}))$ owing to \eqref{stab-id}, we infer that
\begin{align}
(\nabla\SCAL(\ti{\bPi}{}^{\bsg}_T(\ul{v})-\bsg), \tilde{q})_T &= -\tau_T (S_{\dt}^{\rm{mo}}(\ti{\Pi}^v_T(\ul{v}),\Pi^k_{\dT}(v|_{\dt})), S_{\dt}^{\rm{mo}}(0,\Pi^k_{\dt}(\tilde{q}|_{\dt})))_{\dt}\nonumber \\
&= ((\ti{\bPi}{}^{\bsg}_T(\ul{v})-\bsg)\SCAL\bn_T,\Pi^k_{\dt}(\tilde{q}|_{\dt}))_{\dt},\label{5.8}
\end{align}
invoking \eqref{hdg:c} with $\mu := -\Pi^k_{\dt}(\tilde{q}|_{\dt})$ in the last step. 
An integration by parts, \eqref{5.8}, and the $L^2$-orthogonality of $\Pi^k_{\dt}$ lead to
\begin{align}
(\ti{\bPi}{}^{\bsg}_T(\ul{v})-\bsg, \nabla \tilde{q})_T &= -(\nabla\SCAL(\ti{\bPi}{}^{\bsg}_T(\ul{v})-\bsg), \tilde{q})_T+ ((\ti{\bPi}{}^{\bsg}_T(\ul{v})-\bsg)\SCAL \bn_T, \tilde{q})_{\dt}\nonumber\\
&=((\bPi^k_{\dt}(\bsg|_{\dt})-\ti{\bPi}{}^{\bsg}_T(\ul{v}))\SCAL\bn_T,\tilde{q})_{\dt} + ((\ti{\bPi}{}^{\bsg}_T(\ul{v})-\bsg)\SCAL \bn_T, \tilde{q})_{\dt}\nonumber\\
& = ((\bPi^k_{\dt}(\bsg|_{\dt})-\bsg)\SCAL \bn_T, \tilde{q})_{\dt}.\label{4.6}
\end{align}
Altogether, the identities  \eqref{hdg:a}-\eqref{hdg:c} and \eqref{4.6} imply that $(\ti{\bPi}{}_T^{\bsg}(\ul{v}),\ti{\Pi}_T^v(\ul{v}))$ satisfies the  equations \eqref{eq:def2_proj}. Since \eqref{eq:def2_proj} has a unique solution owing to Lemma~\ref{lem:5.1} below, we infer that $(\ti{\bPi}{}_T^{\bsg}(\ul{v}),\ti{\Pi}_T^v(\ul{v}))=(\bPi_T^{\bsg}(\ul{v}),\Pi_T^v(\ul{v}))$. Notice though that the definition \eqref{hdg} needs $\bH(\nabla\cdot;T)$ regularity, which is not needed in the definition \eqref{eq:def2_proj}.  
\end{remark}

\section{Main results}\label{sec:main}

In this section, we present and comment the main results of the paper.

\subsection{Main results on the H-interpolation operator}

We show that the H-interpolation operator is well defined and satisfies superclose or optimal approximation estimates. Finally, we relax the regularity  assumption on the divergence of the  dual variable. The proofs of the results stated in this section are postponed to Section~\ref{sec:H-Int}.

\subsubsection{Well-posedness and approximation estimates}
\label{sec:H-wellp}

\begin{lemma}[Well-posedness]\label{lem:5.1}
For all $\ul{v}:=(\bsg,v)\in \bH^{\nu_{\bsg}}(T)\times H^{\nu_v}(T)$ with $\nu_{\bsg}>\frac{1}{2}$ and  $\nu_v>\frac{1}{2}$, the $\rm{H}$-interpolation operator $\ul{\Pi}^{\sH}_T$ is well-defined.
\end{lemma}

\begin{theorem}[Approximation estimates]\label{prop:A1}
For all $\uv :=(\bsg,v)\in (\bH^{\nu_{\bsg}}(T)\cap \bH(\nabla\cdot;T))\times H^{1}(T)$ with $\nu_{\bsg}>\frac{1}{2}$, the following estimates hold:
\begin{subequations}\label{est}
\begin{align}
\norm{\pw_T(\uv) - \Pi^{k'}_T(v)}_T &\lesssim \ell_\Omega^{-1}h_T^2\norm{\nabla{\cdot}\bsg - \Pi^{k-1}_T(\nabla{\cdot}\bsg)}_T+h_T\norm{\nabla(v-{\cal{E}}_T^{k+1}(v))}_T,\label{est:pw}\\ 
\norm{\ps_T(\uv)-\bPi^k_T(\bsg)}_{T}&\lesssim h_T^{\frac{1}{2}}\norm{(\bPi^k_T(\bsg)-\bsg){\cdot}\bn_T}_{\dt}+h_T\norm{\nabla{\cdot}\bsg - \Pi^{k-1}_T(\nabla{\cdot}\bsg)}_T\nonumber\\&\quad+\ell_\Omega\norm{\nabla(v-{\cal{E}}_T^{k+1}(v))}_T,\label{est:ps} 
\end{align}
\end{subequations}
with the convention that  $\Pi^{k-1}_T(\nabla{\cdot}\bsg):= 0$ for $k=0$.
\end{theorem}
\begin{remark}[Supercloseness/optimality]\label{rem:sup}
Theorem~\ref{prop:A1} gives a supercloseness (resp. optimal) estimate for the primal variable in the equal- (resp. mixed-order) case, and an optimal estimate for the dual variable in both equal- and mixed-order cases. In particular,
when $\uv:=(\bsg,v)\in \bH^{k+1}(\Omega)\times H^{k+2}(\Omega)$, the estimate~\eqref{est:pw} followed by the classical bounds from \eqref{L2:proj-est} and \eqref{est:elliptic} gives 
\begin{align*}
\norm{\pw_T(\uv) - \Pi^{k'}_T(v)}_T \lesssim \ell_\Omega^{-1}h_T^{k+2}(|\bsg|_{\bH^{k+1}(\Omega)}+\ell_\Omega|v|_{H^{k+2}(\Omega)}).
\end{align*}
This is the key estimate to achieve an improved $L^2$-error estimate on the time-integrated primal variable (see Theorem~\ref{thm:tim-int-err} below).
\end{remark}

\begin{remark}[Approximation estimates]\label{rem:5.5}
For all $\uv :=(\bsg,v)\in (\bH^{\nu_{\bsg}}(T)\cap \bH(\nabla\cdot;T))\times H^{1}(T)$ with $\nu_{\bsg}>\frac{1}{2}$, the triangle inequality and Theorem~\ref{prop:A1} provide the following estimates:
\begin{subequations}\label{approx-est}
\begin{align}
\norm{v-\pw_T(\uv)}_T  &\lesssim 	\norm{v - \Pi^{k'}_T(v)}_T + \ell^{-1}_\Omega h_T^2\norm{\nabla{\cdot}\bsg - \Pi^{k-1}_T(\nabla{\cdot}\bsg)}_T+h_T\norm{\nabla(v-{\cal{E}}_T^{k+1}(v))}_T,\label{approx-est.a}\\ 
\norm{\bsg-\ps_T(\uv)}_T&\lesssim \norm{\bsg-\bPi^k_T(\bsg)}_{T}+ h_T^{\frac{1}{2}}\norm{(\bPi^k_T(\bsg)-\bsg){\cdot}\bn_T}_{\dt}+ h_T\norm{\nabla{\cdot}\bsg - \Pi^{k-1}_T(\nabla{\cdot}\bsg)}_T\nonumber\\&\quad+\ell_\Omega\norm{\nabla(v-{\cal{E}}_T^{k+1}(v))}_T.\label{approx-est.b}
\end{align}
\end{subequations}
\end{remark}	

\begin{remark}[Separation of primal and dual variables]
It is interesting to observe that there is no relation between $v$ and $\bsg$ in the definition \eqref{H:proj} and in the estimates  \eqref{est} from Theorem~\ref{prop:A1}.  Thus, the estimate \eqref{est:pw} is equivalent to the two estimates
\begin{align*}
\norm{\pw_T(\bold{0},v)-\Pi^{k'}_T(v)}_T  &\lesssim h_T\norm{\nabla(v-{\cal{E}}_T^{k+1}(v))}_T,\\
\norm{\pw_T(\bsg,0)}_T  &\lesssim \ell^{-1}_\Omega h_T^2\norm{\nabla{\cdot}\bsg - \Pi^{k-1}_T(\nabla{\cdot}\bsg)}_T,
\end{align*}
and the estimate \eqref{est:ps} to
\begin{align*}
\norm{\ps_T(\bsg,0)-\bPi^k_T(\bsg)}_T&\lesssim  h_T^{\frac{1}{2}}\norm{(\bPi^k_T(\bsg)-\bsg){\cdot}\bn_T}_{\dt}+ h_T\norm{\nabla{\cdot}\bsg - \Pi^{k-1}_T(\nabla{\cdot}\bsg)}_T,\\
\norm{\ps_T(\bold{0},v)}_T&\lesssim \ell_\Omega\norm{\nabla(v-{\cal{E}}_T^{k+1}(v))}_T.
\end{align*}
\end{remark}
\begin{remark}[Relaxed regularity]
The approximation estimates in \cite[Proposition~4.6]{du2019invitation} are proved assuming $\bsg\in \bH^1(T)$, whereas we only assume $\bsg \in \bH^{\nu_{\bsg}}(T)\cap \bH(\nabla\cdot;T)$ with $\nu_{\bsg}>\frac{1}{2}$. We show how to relax the regularity on $\bH(\nabla\cdot;T)$ next.
\end{remark}

\subsubsection{Extension to weaker regularity of dual variable}
\label{sec:H-weakregularity}

Now, we relax the regularity assumption $\bsg\in\bH(\nabla{\cdot};T)$ on any simplex $T\in\cT$, required in Theorem~\ref{prop:A1}. Set $\bH^{s-1}(\nabla\cdot;T):=\{\bsg\in\bL^2(T) : \nabla{\cdot}\bsg \in H^{s-1}(T)\}$ with $s\in(\frac{1}{2},1]
$. Notice that, for $s=1$, we have $\bH^{0}(\nabla\cdot;T) = \bH(\nabla\cdot;T)$.
We define the dual norm
\begin{align}
\|\nabla{\cdot}\bsg\|_{H^{s-1}(T)} := \sup_{w\in H^{1-s}_0(T)\setminus\{0\}} \frac{\langle \nabla{\cdot}\bsg, w\rangle_T}{|w|_{H^{1-s}(T)}},\label{def-dualnorm}
\end{align}
where $\langle {\cdot}, {\cdot}\rangle_T$ denotes the
duality pairing between $H^{s-1}(T)$ and $H_0^{1-s}(T)$. Notice that
$1-s <\frac{1}{2}$, so that $H^{1-s}_0(T)=H^{1-s}(T)$ (see e.g., \cite[Theorem~3.19]{ErnGuermondI}). We also
observe that $\Pi^{k-1}_T(\nabla{\cdot}\bsg)$ remains well-defined even if
$\bsg \in \bH^{s-1}(\nabla{\cdot};T)$. 
\begin{theorem}[Approximation estimates with negative regularity on $\nabla{\cdot}\bsg$]\label{thm:relax-reg}
For all $\uv :=(\bsg,v)\in (\bH^{\nu_{\bsg}}(T)\cap \bH^{s-1}(\nabla\cdot;T))\times H^{1}(T)$ with $\nu_{\bsg}>\frac{1}{2}$ and $s\in(\frac{1}{2},1]$, the following estimate holds:
\begin{subequations}\label{est:relax}
\begin{align}
\norm{\pw_T(\uv) - \Pi^{k'}_T(v)}_T &\lesssim \ell_\Omega^{-1}h_T^{1+s}\norm{\nabla{\cdot}\bsg - \Pi^{k-1}_T(\nabla{\cdot}\bsg)}_{H^{s-1}(T)}+h_T\norm{\nabla(v-{\cal{E}}_T^{k+1}(v))}_T,\label{est:pw-relax}\\ 
\norm{\ps_T(\uv)-\bPi^k_T(\bsg)}_{T}&\lesssim h_T^{\frac{1}{2}}\norm{(\bPi^k_T(\bsg)-\bsg){\cdot}\bn_T}_{\dt}+h_T^s\norm{\nabla{\cdot}\bsg - \Pi^{k-1}_T(\nabla{\cdot}\bsg)}_{H^{s-1}(T)}\nonumber\\&\quad+\ell_\Omega\norm{\nabla(v-{\cal{E}}_T^{k+1}(v))}_T.\label{est:ps-relax} 
\end{align}
\end{subequations}
\end{theorem}

\subsection{A priori error analysis for the wave equation}
This section presents the energy and $L^2$-error estimates for the wave equation. We notice that our main focus is on the $L^2$-error estimate, whose proof does not use the energy-error estimate. We present the latter to showcase the improvement achieved by the former. The energy-error estimate is also of independent interest. The proofs of the results stated in this section are postponed to Section~\ref{sec:errest}.

\subsubsection{Errors and their equations}

Let $\uv:=(\bsg,v)\in C^0(\ol{J};\ul{V})\cap C^1(\ol{J};\ul{L})$ solve the continuous problem \eqref{weak-form} and let $(\bsg_\cT,\h{v}_\cM)\in C^1(\ol{J};\bS_\cT^k\times\h{V}_{\cM 0}^{k})$ solve the discrete problem \eqref{eq:HHO}. Let $\ul{D}^{\sH}:= \bH^{\nu_{\bsg}}(\Omega)\times H^{\nu_v}(\Omega)$ with $\nu_{\bsg}>\frac{1}{2}$ and $\nu_v>\frac{1}{2}$ be the domain of the global H-interpolation operator $(\bPi^{\bsg}_\cT,\Pi^v_\cT)$, which is defined componentwise as $ \bPi^{\bsg}_\cT({\cdot})|_T := \bPi^{\bsg}_T({\cdot}|_T)$ and  $\Pi^v_\cT({\cdot})|_T := \Pi^v_T({\cdot}|_T)$ for all $T\in\cT$ with $(\bPi^{\bsg}_T,\Pi^v_T)$ defined by \eqref{H:proj}, i.e., by \eqref{H:int} in the equal-order case and by \eqref{eq:def2_proj} in the mixed-order case. Thus, $(\bPi^{\bsg}_\cT,\Pi^v_\cT):\ul{D}^{\sH}\to \bS_\cT^k\times V_{\cT}^{k'}$. 

We define the space semi-discrete errors, for all $t\in \ol{J}$, as
\begin{alignat*}{4}
&\text{H-errors:}
&&\quad\bet_\cT^{\sH} (t) &&:= \bsg_\cT(t) - \bPi^{\bsg}_\cT(\ul{v}(t)), &&\qquad \h{e}_\cM^{\sH}(t):= \h{v}_\cM(t)-(\Pi^v_\cT(\uv(t)),\Pi_\cF^k(v(t)|_\cF)),
\\
&\text{HHO-errors:}
&&\quad\bet_\cT^{\sPi} (t) &&:= \bsg_\cT(t) - \bPi^{k}_\cT(\bsg(t)),
&&\qquad \h{e}_\cM^{\sPi}(t):= \h{v}_\cM(t)-(\Pi^{k'}_\cT(v(t)),\Pi^k_\cF(v(t)|_\cF)),
\end{alignat*}	
where $(\bPi^k_{\cT},\Pi^{k'}_\cT):\ul{L}\to \bS^k_{\cT}\times V_{\cT}^{k'}$ are the global $L^2$-orthogonal projections.

The error equations associated with the above errors have different properties. 
Interestingly, the combination of the two 
helps in proving the improved $L^2$-error estimates.

\begin{lemma}[Error equations using H-interpolation]\label{lem:6.1}
Assume that $\uv\in  C^1(\ol{J};\ul{D}^{\sH})$. The following holds: For all $t\in \ol{J}$ and all $(\bxi_\cT,\h{w}_\cM)\in \bS_\cT^k\times\h{V}_{\cM 0}^{k}$,
\begin{subequations} \label{h-error}
\begin{align}
&(\pt\bet_\cT^{\sH}(t),\bxi_\cT)_{\Omega} - (\bG_\cT(\h{e}_\cM^{\sH}(t)),\bxi_\cT)_{\Omega}=(\pt(\bsg(t)-\bPi_\cT^{\bsg}(\uv(t))),\bxi_\cT)_{\Omega},\label{h-error:1a}\\
&(\pt e_\cT^{\sH}(t),w_\cT)_{\Omega} +(\bet_\cT^{\sH}(t), \bG_\cT(\h{w}_\cM))_{\Omega} +s_{\cM}(\h{e}_\cM^{\sH}(t),\h{w}_\cM) = (\partial_t(v(t)-\Pi^v_\cT(\ul{v}(t))),w_\cT)_{\Omega}.\label{h-error:1b}
\end{align}
\end{subequations}
\end{lemma}

\begin{lemma}[Error equations using HHO interpolation] \label{lem:err_HHO}
The following holds: For all $t\in \ol{J}$ and all  $(\bxi_\cT,\h{w}_\cM)\in \bS_\cT^k\times\h{V}_{\cM 0}^{k}$,
\begin{subequations} 
\begin{align}
	&(\pt\bet_\cT^{\sPi}(t),\bxi_\cT)_{\Omega} - (\bG_\cT(\h{e}_{\cM}^{\sPi}(t)),\bxi_\cT)_{\Omega}=0,\label{hho-error:1a}\\
	&(\pt e_\cT^{\sPi}(t),w_\cT)_{\Omega} +(\bet_\cT^{\sPi}(t), \bG_\cT(\h{w}_{\cM}))_{\Omega} +s_{\cM}(\h{e}_{\cM}^{\sPi}(t),\h{w}_{\cM}) = \psi_{\cM}^{\sPi}(\ul{v}(t);\h{w}_{\cM}),\label{hho-error:1b}
\end{align}
\end{subequations}
where the linear form $\psi_{\cM}^{\sPi}(\ul{z};\cdot):\h{V}_{\cM 0}^k\to\rbb$ is defined, for all $\ul{z}:=(\bze,z)\in\ul{V}$ and $ \h{w}_\cM\in \h{V}_{\cM 0}^{k}$, as
\begin{align}
\psi_{\cM}^{\sPi}(\ul{z};\h{w}_\cM):=-(\nabla{\cdot}\bze,w_{\cT})_\Omega - (\bPi^k_{\cT}(\bze(t)),\bG_{\cT}(\h{w}_{\cM}))_\Omega - s_{\cM}((\Pi^{k'}_\cT(z),\Pi^k_\cF(z(t)|_\cF)),\h{w}_\cM).\label{cons-HHO}
\end{align}
\end{lemma}

\begin{corollary}[Combination of discrete HHO- and H-errors] \label{cor:comb_err}
Assume that  $\uv\in  C^1(\ol{J};\ul{D}^{\sH})$. The following holds: For all $t\in \ol{J}$ and all  $(\bxi_\cT,\h{w}_\cM)\in \bS_\cT^k\times\h{V}_{\cM 0}^{k}$,
\begin{subequations}\label{h:error-new}
\begin{align}
&(\pt\bet_\cT^{\sPi}(t),\bxi_\cT)_{\Omega} - (\bG_\cT(\h{e}_\cM^{\sH}(t)),\bxi_\cT)_{\Omega}=0,\label{h:error-new.a}\\
&(\pt e_\cT^{\sPi}(t),w_\cT)_{\Omega} +(\bet_\cT^{\sH}(t), \bG_\cT(\h{w}_\cM))_{\Omega} +s_{\cM}(\h{e}_\cM^{\sH}(t),\h{w}_\cM) = 0.\label{h:error-new.b}
\end{align}
\end{subequations}
\end{corollary}

\begin{corollary}[Initial error on time-derivatives] \label{cor:init-dt}
The following holds:
\begin{align} 
\pt\bet_\cT^{\sPi}(0) = \bold{0},\quad	\pt e_\cT^{\sPi}(0) = 0.\label{IC:time-diff}
\end{align}
\end{corollary}

\begin{remark}[Reconstructed error gradient]
The equations \eqref{hho-error:1a}  and \eqref{h:error-new.a} imply that $\bG_\cT(\h{e}_{\cM}^{\sPi}(t)) = \bG_\cT(\h{e}^{\sH}_{\cM}(t))$ for all $t\in \ol{J}$.
\end{remark}

\begin{remark}[Ritz projection]
Lemma~\ref{lem:6.1} proposes a way to construct a global Ritz projection $(\bPi^{\sR_{\bsg}}_{\cT},\h{\Pi}^{\sR_v}_{\cM})$ satisfying the following equations: Given $\uv:=(\bsg,v)\in \ul{V}$, find $(\bPi^{\sR_{\bsg}}_{\cT}(\uv),\h{\Pi}^{\sR_v}_{\cM}(\uv))\in \bS_\cT^k\times \h{V}_{\cM 0}^{k}$ such that, for all $(\bxi_\cT,\h{w}_\cM)\in \bS_\cT^k\times \h{V}_{\cM 0}^{k}$,
\begin{subequations}\label{eq:Ritz}
\begin{align}
(\bG_\cT(\h{\Pi}^{\sR_v}_{\cM}(\uv)),\bxi_\cT)_\Omega &= (\nabla v,\bxi_\cT)_\Omega\\
(\bPi^{\sR_{\bsg}}_{\cT}(\uv),\bG_\cT(\h{w}_{\cM}))_\Omega +s_{\cM}(\h{\Pi}^{\sR_v}_{\cM}(\uv)),\h{w}_{\cM}) &= -(\nabla{\cdot}\bsg,w_{\cT})_\Omega.
\end{align}
\end{subequations}
One can verify that $(\bPi_{\cT}^{\bsg}(\uv),\bPi_{\cT}^v(\uv),\Pi_{\cF}^k(v|_{\cF}))$ satisfies the above equations, which implies the existence of a Ritz projection. However, the uniqueness in \eqref{eq:Ritz} is not guaranteed. Indeed, we have  $\h{\Pi}^{\sR_v}_{\cM}(\ul{0}) =0,$ but $\bPi_{\cT}^{\sR_{\bsg}}(\ul{0})$ may not be zero. We only have  $\nabla{\cdot}\bPi_{\cT}^{\sR_{\bsg}}(\ul{0}) = 0$ and $\sum_{F\in\cF^i}\jump{\bPi_{\cT}^{\sR_{\bsg}}(\ul{0})}{\cdot}\bn_F=0$, implying that there may not be an unique solution to \eqref{eq:Ritz}. 
\end{remark}

\subsubsection{Error estimates}
Let $t\in \ti{J}:=\ol{J}\setminus\{0\}=(0,T_f]$ and set $J_t:=(0,t)$. We use the following shorthand notation:
\begin{alignat*}{2} |\psi|_{L^p(J_t;**)}^{{p}}&:=\int_{J_t}|\psi(s)|_{**}^p\,\ds\quad\text{for}\;p\in[1,\infty),\qquad &|\psi|_{C^0(J_t;**)}&:=\sup_{s\in \ol{J_t}}|\psi(s)|_{**},\\
\norm{\psi}_{L^p(J_t;**)}^{{p}}&:=\int_{J_t}\norm{\psi(s)}_{**}^p\,\ds\quad\text{for}\;p\in[1,\infty),\qquad &\norm{\psi}_{C^0(J_t;**)}&:=\sup_{s\in\ol{J_t}}\norm{\psi(s)}_{**},
\end{alignat*}
where the seminorm  $|\cdot|_{**}$  and the norm  $\norm{\cdot}_{**}$ depend on the context.
For an integer number $m\in\{1,2\}$, we also consider the seminorm
$|\psi|_{C^m(J;**)}:=\sum_{r\in\{0{:}m\}}T_f^{r}|\partial_t^{r}\psi|_{C^0(J;**)}$
and the norm $\|\psi\|_{C^m(J;**)}:=\sum_{r\in\{0{:}m\}}T_f^{r}\|\partial_t^{r}\psi\|_{C^0(J;**)}$
with the convention $\partial_t^0\psi:=\psi$.

We denote the stabilization seminorm by $|\cdot|_{\sS}^2 := s_{\cM}(\cdot,\cdot)$. For a linear form $\psi:\h{V}^k_{\cM 0}\to\rbb$, we define the dual norm
\begin{align} 
\label{7.2}
\norm{\psi}_{(\HHO)'}:=\sup_{\h{w}_{\cM}\in\h{V}^{k}_{\cM 0}\setminus\{0\}}\frac{\psi(\h{w}_{\cM})}{\norm{\h{w}_{\cM}}_{\HHO}}\quad \text{with}\quad	\norm{\h{w}_{\cM}}_{\HHO}^2:=\sum_{T\in\cT}\Big(\norm{\nabla w_T}_{T}^2+h_T^{-1}\norm{\delta_{\dt}(\h{w}_T)}^2_{\dt}\Big),
\end{align} 
and recalling the linear form $\psi_{\cM}^{\HHO}(\ul{z};\cdot):\h{V}^k_{\cM 0}\to\rbb$ with $\ul{z}\in\ul{V}$ defined in \eqref{cons-HHO}, we define for all $\ul{v}\in C^0(\ol{J};\ul{V})$ and all $p\in[1,\infty)$,
\begin{alignat*}{2}
\norm{\psi_{\cM}^{\HHO}(\ul{v};\cdot)}_{L^p(J_t;(\HHO)')}^{{p}}&:=\int_{J_t}\norm{\psi_{\cM}^{\HHO}(\ul{v}(s);\cdot)}_{(\HHO)'}^p\,\ds,\; \\
\norm{\psi_{\cM}^{\HHO}(\ul{v};\cdot)}_{C^0(J_t;(\HHO)')}&:=\sup_{s\in\ol{J_t}}\norm{\psi_{\cM}^{\HHO}(\ul{v}(s);\cdot)}_{(\HHO)'}.
\end{alignat*}

For dimensional consistency, we keep track of the dependency of the constants on $T_f$ and $\ell_\Omega$, but we hide the non-dimensional ratio $\ell_\Omega/T_f$ in the constants.

\begin{theorem}[Optimal energy-error estimate]\label{thm:energy-error}
Assume that $\underline{v}_0:=(\bsg_0,v_0) \in \underline{D}^{\sH}$. Let $\ul{v}:=(\bsg,v)\in C^0(\ol{J};\ul{V})\cap C^1(\ol{J};\ul{L})$  
solve the continuous problem \eqref{weak-form} and let $(\bsg_\cT,\h{v}_\cM)\in C^1(\ol{J};\bS_\cT^k\times\h{V}_{\cM 0}^{k})$ solve the discrete problem \eqref{eq:HHO} with the initial condition 
\begin{align}
\label{def:initc}
\bsg_\cT(0) = \bPi^{\bsg}_\cT(\uv_0),\quad  v_\cT(0) = \Pi^v_\cT(\uv_0)\quad \text{with}\quad \uv_0:=(\bsg_0,v_0).
\end{align}  
Assume that $\ul{v}\in C^1(\ol{J};\ul{V})$.
Then the following holds: For all $t\in \ti{J}$,
\begin{align}
&\norm{e_\cT^{\sPi}}_{C^0(J_t;\Omega)}+\norm{\bet^{\sPi}_\cT}_{C^0(J_t;\Omega)} + |\h{e}^{\sPi}_\cM|_{L^2(J_t;\sS)} \lesssim \norm{\bsg_0-\bPi^{\bsg}_\cT(\ul{v}_0))}_{\Omega}+\norm{v_0-\Pi^v_\cT(\ul{v}_0)}_{\Omega}\nonumber\\&\hspace{5cm}+\norm{\psi_{\cM}^{\sPi}(\uv;\cdot)}_{L^{\infty}(J_t;(\sPi)')}+\norm{\psi_{\cM}^{\sPi}(\pt \uv;\cdot)}_{L^{1}(J_t;(\sPi)')}.\label{energy-err.a}
\end{align}	
In addition, if $(\bsg,v)\in  C^2(\ol{J};\ul{V})$, we have for all $t\in \ti{J}$,
\begin{align}
\norm{\partial_te^{\sPi}_\cT}_{C^0(J_t;\Omega)}+\norm{\partial_t\bet^{\sPi}_\cT}_{C^0(J_t;\Omega)} + |\partial_t\h{e}^{\sPi}_\cM|_{L^2(J_t;\sS)}  &\lesssim \norm{\psi_{\cM}^{\sPi}(\pt \uv;\cdot)}_{L^{\infty}(J_t;(\sPi)')}\nonumber\\
&\quad+\norm{\psi_{\cM}^{\sPi}(\partial_{tt} \uv;\cdot)}_{L^{1}(J_t;(\sPi)')}.\label{energy-err.b}
\end{align}		
\end{theorem}

\begin{remark}[Convergence rates]
Assuming $\ul{z}:=(\bze,z)\in \bH^{\nu_{\bsg}}(\Omega)\times H^1_0(\Omega)$ with $\nu_{\bsg}>\frac{1}{2}$, the consistency error $\psi_{\cM}^{\HHO}$ defined in \eqref{cons-HHO} can be further simplified as 
\begin{align*}
\psi_{\cM}^{\sPi}(\ul{z};\h{w}_\cM)=\sum_{T\in\cT}((\bze-\bPi_\cT^k(\bze)){\cdot}\bn_T,w_{\dt}-w_T)_{\dt} - s_{\cM}((\Pi^{k'}_\cT(z),\Pi^k_\cF(z|_\cF)),\h{w}_\cM),
\end{align*}
and one can show that
\[\norm{\psi_{\cM}^{\HHO}(\ul{z};\cdot)}_{(\HHO)'}\lesssim \Big\{\sum_{T\in\cT} h_T\norm{(\bze-\bPi_\cT^k(\bze)){\cdot}\bn_T}_{\dt}^2+\ell_\Omega\norm{\nabla(z-{\cal{E}}_T^{k+1}(z))}_T^2\Big\}^{\frac{1}{2}}.\]
Hence, if $\ul{v}:=(\bsg,v)\in  C^1(\ol{J}; \bH^{\nu_{\bsg}}(\Omega)\times H^{1+\nu_{v}}(\Omega))$ with $\nu_{\bsg},\nu_v\in(\frac{1}{2},k+1]$, we have
\begin{multline*}
\norm{e_\cT^{\sPi}}_{C^0(J_t;\Omega)}+\norm{\bet^{\sPi}_\cT}_{C^0(J_t;\Omega)} + |\h{e}^{\sPi}_\cM|_{L^2(J_t;\sS)}\lesssim 
h^{\nu_{\bsg}}|\bsg|_{C^1(\ol{J};\bH^{\nu_{\bsg}}(\Omega))}+\ell_\Omega h^{\nu_v}|v|_{C^1(\ol{J};H^{1+\nu_{v}}(\Omega))}.
\end{multline*}
Moreover, if $\ul{v}\in  C^2(\ol{J}; \bH^{\nu_{\bsg}}(\Omega)\times H^{1+\nu_{v}}(\Omega))$ with $\nu_{\bsg},\nu_v\in(\frac{1}{2},k+1]$, we have
\begin{multline*}
\norm{\partial_te^{\sPi}_\cT}_{C^0(J_t;\Omega)}\!+\!\norm{\partial_t\bet^{\sPi}_\cT}_{C^0(J_t;\Omega)}\! +\! |\partial_t\h{e}^{\sPi}_\cM|_{L^2(J_t;\sS)}\lesssim 
h^{\nu_{\bsg}}|\bsg|_{C^2(\ol{J};\bH^{\nu_{\bsg}}(\Omega))}+\ell_\Omega h^{\nu_v}|v|_{C^2(\ol{J};H^{1+\nu_{v}}(\Omega))}.
\end{multline*}
Whenever $\nu_{\bsg}=\nu_v=k+1$, the above estimates converge optimally with rate $\mathcal{O}(h^{k+1})$.
\end{remark}
Recall that $J_t:=(0,t)$ for all $t\in \ti{J}$ and set 
$\ut{e}^{\sPi}_\cT(t):=\int_{J_t} e_\cT^{\sPi}(s)\,\ds,\;
\forall t\in \ti{J}.$
 
\begin{theorem}[Superconvergent $L^2$-error estimate for time-integrated primal variable]\label{thm:tim-int-err}
Let $s\in(\frac{1}{2},1]$ be the index of elliptic regularity pickup in $\Omega$. Let $\delta := s$ if $k'=0$ and  $\delta :=0$ if $k'\ge1$. 
Assume that $\underline{v}_0\in \underline{D}^{\sH}$ and $\ul{v}\in C^1(\ol{J};\ul{V})\cap C^2(\ol{J};\ul{L})$. The following holds:
\begin{align}
\norm{\ut{e}_\cT^{\sPi}(T_f)}_\Omega \lesssim {}&  
h^s \ell_\Omega^{1-s} T_f^{\frac{1}{2}} \Big\{ \ell_\Omega^{-\frac12} \abs{\h{e}_{\cM}^{\sPi}}_{L^2(J;\sS)} + \ell_\Omega^{\frac12} \abs{\pt\h{e}_{\cM}^{\sPi}}_{L^2(J;\sS)} + \norm{\pt\bet_\cT^{\sPi}}_{L^2(J;\bL^2(\Omega))} \nonumber \\
&+ \ell_\Omega \norm{\nabla (\pt v - {\cal E}_\cT^{k+1} (\pt v))}_{L^2(J;\bL^2(\Omega))}
+ \norm{\pt\bsg-\bPi^k_{\cT}(\pt\bsg)}_{L^2(J;\bL^2(\Omega))} \nonumber \\
&+h^{1-\delta}\ell_\Omega^\delta \norm{\nabla{\cdot}\pt\bsg-\Pi_\cT^{k'}(\nabla{\cdot}\pt\bsg)}_{L^2(J;L^2(\Omega))} \Big\} + \ell_\Omega\norm{\Pi_\cT^{k'}(v_0)-\Pi^v_\cT(\ul{v}_0)}_{\Omega}.
\end{align}
\end{theorem}

\begin{remark}[Convergence rates]
Assume $\ul{v}\in  C^2(\ol{J}; \bH^{\nu_{\bsg}}(\Omega)\times H^{1+\nu_{v}}(\Omega))$ with $\nu_{\bsg},\nu_v\in(\frac{1}{2},k+1]$. Let $r:=\max(\nu_{\bsg}-1,0)$ if $k'\ge1$ and $r:=\nu_{\bsg}-1+s\ge0$ if  $k'=0$. Assume that $\nabla{\cdot}\pt\bsg \in C^0(\ol{J};H^{r}(\Omega))$ (this follows from $\bsg \in C^1(\ol{J};\bH^{\nu_{\bsg}}(\Omega))$ whenever $k'\ge 1 $ and $\nu_{\bsg}\ge1$). Then we have
\begin{align*}
\norm{\ut{e}_\cT^{\sPi}(T_f)}_\Omega &\lesssim h^s\ell_\Omega^{-s}T_f 
\Big\{h^{\nu_{\bsg}}|\bsg|_{C^2(\ol{J};\bH^{\nu_{\bsg}}(\Omega))}+\ell_\Omega h^{\nu_v}| v|_{C^2(\ol{J};H^{1+\nu_v}(\Omega))}\\& + h^{1+r-\delta} \ell_\Omega^\delta  |\nabla{\cdot}\pt\bsg|_{C^0(\ol{J};H^r(\Omega))}\Big\}+h\big\{h^{\nu_{\bsg}}|\bsg_0|_{\bH^{\nu_{\bsg}}(\Omega)}+\ell_\Omega h^{\nu_v}| v_0|_{H^{1+\nu_v}(\Omega)}\},
\end{align*}
where the bound on the last term involving the initial condition follows from \eqref{est:pw}. Notice that $1+r-\delta \geq \nu_{\bsg}$ for $k'\geq 1$ and $1+r-\delta = \nu_{\bsg}$ for $k'=0$. Thus, in all cases for $k'$, the term involving $|\nabla{\cdot}\pt\bsg|_{C^0(\ol{J};H^r(\Omega))}$ converges at least at the same rate as the other two terms between braces.
Altogether, when $\nu_{\bsg}=\nu_v =k+1$ and $s=1$, the above estimate converges at rate $\mathcal{O}(h^{k+2})$.
\end{remark}

\begin{remark}[$k'=0$]
The modification in the analysis for $k'=0$ whereby some slight additional
regularity on $\nabla{\cdot}\pt\bsg$ is required, is inspired from the $L^2$-error analysis
of HHO methods in the elliptic case (see \cite[Lem.~2.11]{cicuttin2021hybrid}). This 
is the first time this idea is used in the time-dependent context.
\end{remark}

\begin{remark}[Time-integrated primal variable]
In the context of the second-order formulation in time of the wave equation, the time-integrated primal variable has a physical meaning. For instance, it can be intepreted as a displacement when the primal variable represents a velocity. In this case, Theorem~\ref{thm:tim-int-err} provides an improved $L^2$-error bound on the displacement. To evaluate this error, it can be more computationally effective, as in \cite{cockburn2014uniform}, to evolve in time only the cell mean-values of the discrete velocity and use the discrete gradient at final time.
\end{remark}

\section{Preliminary results}
\label{sec:aux}

In this section, we state and prove the auxiliary
results that we need to prove our main results. We gather them in three groups. The first is about general polynomials. The second is about  standard HHO results, and the third is about some novel results on the HHO stabilization.

In what follows, the inequality $a\leq Cb$ for positive numbers $a$ and $b$ is often abbreviated as $a\lesssim b$. The value of $C$ can be different at each occurence provided it is independent of the parameters $\ell_\Omega$, $T_f,$ and the mesh-size $h$; the value can depend on the mesh shape-regularity, the polynomial degree, and the space dimension. We also write  $a \approx b$ whenever $a\lesssim b\lesssim a$. 

\subsection{Two results on polynomials}
\begin{lemma}[Bound on vector-valued polynomials]\label{lem:3.5}
For all $T\in\cT$, fixing any  $d$ faces $\{F_i\}_{i=1,\dots,d}\subset \cF_{\dt}$, the following  holds: For all $\bq\in \pbb^k(T;\rbb^d)$,
\begin{align*}
\norm{\bq}_{T}\lesssim \norm{\bPi_T^{k-1}(\bq)}_T+\sum_{i\in \{1:d\}} h_T^{\frac{1}{2}} \norm{\bq{\cdot}\bn_{T}}_{F_i}, 
\end{align*}
with the convention that $\pbb^{-1}(T;\rbb^d) = \{\bold{0}\}$.
\end{lemma}
\begin{proof}
A scaling argument and \cite[Lemma~2.1]{du2019invitation}  provide the expected estimate.
\end{proof}

For all  $T\in\cT$, we define the space
\begin{align}
\pbb^{k'}_\perp(T;\rbb)&:=\{q\in \pbb^{k'}(T;\rbb):(q,r)_{T}=0\quad\forall r\in\pbb^{k-1}(T;\rbb)\},\label{4.8}
\end{align}
with the convention that $\pbb^{-1}(T;\rbb) = \{0\}$ so that the condition in the above definition is trivial for $k=0$.

\begin{lemma}[Bound on scalar-valued polynomials]\label{lem:A1}
For all $T\in\cT$, the following holds: For all $p\in\pbb_\perp^{k'}(T;\rbb)$,
\begin{subequations}
\begin{empheq}[left={\norm{p}_T\lesssim\empheqlbrace}]{alignat = 2}
& h_T^{\frac{1}{2}} \norm{p}_{\dt}\qquad&\text{for}\; &k'=k,\label{PO.a}\\
&h_T\norm{\nabla p}_T +h_T^{\frac{1}{2}} \norm{p}_{\dt}\qquad&\text{for}\;& k'=k+1.
\end{empheq}
\end{subequations}
\end{lemma}
\begin{proof}
For $k'=k$, the proof is given in \cite[Lemma~A.1]{cockburn2010projection}. Here, we prove the estimate for $k'=k+1$. Let $T\in\cT$ and $p\in\pbb_\perp^{k+1}(T;\rbb)$.
Notice  that $(\Pi^k_T(p),r)_T = (p,r)_T =0 $ for all $r\in\pbb^{k-1}(T;\rbb)$.  Invoking \eqref{PO.a} for $\Pi^k_T(p)$, we infer that $\norm{\Pi^k_T(p)}_T \lesssim h_T^{\frac{1}{2}} \norm{\Pi^k_T(p)}_{\dt}$. This, the  triangle inequality, and the discrete trace inequality lead to
\begin{align}
\norm{p}_T &\leq \norm{p-\Pi^k_T(p)}_T + \norm{\Pi^k_T(p)}_T\nonumber\\&\lesssim \norm{p-\Pi^k_T(p)}_T  +h_T^{\frac{1}{2}} \norm{\Pi^k_T(p)-p}_{\dt}+h_T^{\frac{1}{2}} \norm{p}_{\dt}\nonumber\\&
\lesssim \norm{p-\Pi^k_T(p)}_T+h_T^{\frac{1}{2}} \norm{p}_{\dt}.\nonumber
\end{align}
The observation that $ \norm{p-\Pi^k_T(p)}_T\leq \norm{p-\Pi^0_T(p)}_T\lesssim h_T\norm{\nabla p}_T$ from the $L^2$-othogonality of $\Pi^k_T$ and the Poincar\'e inequality conclude the proof.
\end{proof}

 \subsection{Basic HHO results}
\begin{lemma}[Stability]\label{prop:stability}
For all $\h{v}_T\in\h{V}_T^k$ and all $T\in\cT$, we have
\begin{align*}
\norm{\nabla v_T}_{T}^2+h_T^{-1}\norm{v_T-v_{\dt}}^2_{\dt}\approx\norm{\bG_{T}(\h{v}_T)}_{T}^2+h_T^{-1}\norm{S_{\dt}^{\HHO}(\h{v}_T)}^2_{\dt}.
\end{align*}
\end{lemma} 
\begin{proof}
See \cite[Lemma~4]{DE_2015, DEL_2014}  or \cite[Lemma~2.6]{cicuttin2021hybrid}.
\end{proof}

\begin{lemma}[Bound on the stabilization term using $\Pi^{k'}_T$]\label{lem:3.4}
For all $v\in H^{1}(T)$ and all $T\in\cT$,  we have
\begin{align}
h_T^{-\frac{1}{2}}\norm{S_{\dt}^{\HHO}(\Pi^{k'}_T(v), \Pi^k_{\dt}(v|_{\dt}))}_{\dt}\lesssim \norm{\nabla(v-{\cal{E}}_T^{k+1}(v))}_T.\label{S1}
\end{align}
Consequently, we have, for all $p\in \pbb^{k'}(T;\rbb)$,
\begin{align}
S_{\dt}^{\HHO}(p,\Pi^k_{\dt}(p|_{\dt})) = 0.\label{stab-id}
\end{align}
\end{lemma}
\begin{proof}
For the equal-order case, see \cite[Lemma~3]{DEL_2014} or \cite[Lemma~2.7]{cicuttin2021hybrid}. Let us consider the mixed-order case. It is proved in \cite[Lemma~3.4]{cicuttin2021hybrid} that $h_T^{-\frac{1}{2}}\norm{S_{\dt}^{\rm{mo}}(\Pi^{k+1}_T(v), \Pi^k_{\dt}(v|_{\dt}))}_{\dt}\lesssim \norm{\nabla(v-\Pi_T^{k+1}(v))}_T$. Here, we establish the slightly stronger estimate \eqref{S1}. The definition of the stabilization operator \eqref{HHO-stab:1b} and the $L^2$-stability of $\Pi^k_{\dt}$ lead to 
\begin{align}
\norm{S_{\dt}^{\rm{mo}}(\Pi^{k+1}_T(v), \Pi^k_{\dt}(v|_{\dt}))}_{\dt}  =\norm{\Pi^k_{\dt}((\Pi^{k+1}_T(v)- v)|_{\dt})}_{\dt}\leq \norm{\Pi^{k+1}_T(v)- v}_{\dt}.\label{3.15}
\end{align}
The triangle inequality, the discrete trace inequality, and the multiplicative trace inequality show that
\begin{align}
\norm{\Pi^{k+1}_T(v)- v}_{\dt}  &\leq \norm{\Pi^{k+1}_T(v)- {\cal{E}}_T^{k+1}(v)}_{\dt} + \norm{{\cal{E}}_T^{k+1}(v)-v}_{\dt}\nonumber\\
&\lesssim h_T^{-\frac{1}{2}}\norm{\Pi^{k+1}_T(v)- {\cal{E}}_T^{k+1}(v)}_T + h_T^{-\frac{1}{2}}\norm{{\cal{E}}_T^{k+1}(v)-v}_{T}+h_T^{\frac{1}{2}}\norm{\nabla({\cal{E}}_T^{k+1}(v)-v)}_{T}.\nonumber
\end{align}
The observation $\norm{\Pi^{k+1}_T(v)-{\cal{E}}_T^{k+1}(v)}_{T}\leq \norm{v-{\cal{E}}_T^{k+1}(v)}_T$ from the Pythagoras identity, and  the Poincar\'e inequality result in
\begin{align}
\norm{\Pi^{k+1}_T(v)- v}_{\dt} &\lesssim h_T^{-\frac{1}{2}}\norm{{\cal{E}}_T^{k+1}(v)-v}_{T}+h_T^{\frac{1}{2}}\norm{\nabla({\cal{E}}_T^{k+1}(v)-v)}_{T}\nonumber\\&\lesssim h_T^{\frac{1}{2}}\norm{\nabla({\cal{E}}_T^{k+1}(v)-v)}_{T}.\label{3.16}
\end{align}
The combination of \eqref{3.15}-\eqref{3.16} concludes the proof.
\end{proof}

\subsection{Novel results on HHO stabilization}

\begin{lemma}[Bound on the stabilization term using $\Pi^v_T$]\label{lem4.4}
For all $T\in\cT$ and all $\uv:=(\bsg,v)\in \bH^{\nu_{\bsg}}(T)\times H^{1}(T)$ with $\nu_{\bsg}>\frac{1}{2}$, the following estimate holds:
\begin{align}
h_T^{-\frac{1}{2}}\norm{S_{\dt}^{\HHO}(\pw_T(\uv),\Pi^k_{\dt}(v|_{\dt}))}_{\dt} \lesssim  h_T^{-1}\norm{\pw_T(\uv) - \Pi^{k'}_T(v)}_{T} + \norm{\nabla(v-{\cal{E}}^{k+1}_T(v))}_T.\label{stab-bound}
\end{align}
\end{lemma}
\begin{proof}
(1) In the equal-order case ($k'=k$), the definition \eqref{HHO-stab:1a} of $S_{\dt}^{\rm{eo}}$ leads to 
\begin{align*}
S_{\dt}^{\rm{eo}}(\pw_T(\uv),\Pi^k_{\dt}(v|_{\dt}))  &= \Pi^k_{\dt}\Big\{\pw_T(\uv)|_{\dt}-\Pi^k_{\dt}(v|_{\dt})+((1-\Pi^k_T)R_T(\pw_T(\uv),\Pi^k_{\dt}(v|_{\dt})))|_{\dt}\Big\}
\\&=\Pi^k_{\dt}\Big\{\pw_T(\uv)|_{\dt}-v|_{\dt}+((1-\Pi^k_T)R_T(\pw_T(\uv),\Pi^k_{\dt}(v|_{\dt})))|_{\dt}\Big\}
\\&= \Pi^k_{\dt}\Big\{(\pw_T(\uv) - \Pi^k_T(v))|_{\dt}+(\Pi^k_T(\eta) - \eta)|_{\dt}\Big\},
\end{align*}
where $\eta:= v- R_T(\pw_T(\uv),\Pi^k_{\dt}(v|_{\dt}))$. The $L^2$-stability of $\Pi^k_{\dt}$ and the triangle inequality imply that
\begin{align}
\norm{S_{\dt}^{\rm{eo}}(\pw_T(\uv),\Pi^k_{\dt}(v|_{\dt}))}_{\dt} &\leq \norm{\pw_T(\uv) - \Pi^k_T(v)}_{\dt}+\norm{\eta-\Pi^k_T(\eta)}_{\dt} \nonumber\\&
\lesssim h_T^{-\frac{1}{2}}\norm{\pw_T(\uv) - \Pi^k_T(v)}_{T}+h_T^{\frac{1}{2}}\norm{\nabla \eta}_T,\label{eq:4.18}
\end{align}
with a discrete trace inequality and the Poincar\'e inequality in the second step.	Notice from the definition of $R_T$ that, for all $q\in \pbb^{k+1}_*(T;\rbb)$,
\begin{align*}
(\nabla R_T(\pw_T(\uv),\Pi^k_{\dt}(v|_{\dt})),\nabla q)_T &= - (\pw_T(\uv), \Delta q)_T + (\Pi^k_{\dt}(v|_{\dt}),\nabla q{\cdot}\bn_T)_{\dt}\\&
=-(v,\Delta q)_T + (v,\nabla q{\cdot}\bn_T)_{\dt}\\&
=(\nabla v,\nabla q)_T,
 \end{align*}
where the second equality follows from the definition \eqref{H:1a} of $\pw_T(\uv)$ since $\Delta q\in \pbb^{k-1}(T;\rbb)$ and the $L^2$-orthogonality of $\Pi^k_{\dt}$ since $\nabla q{\cdot}\bn_T\in\pbb^k(\cF_{\dt};\rbb)$, and the third equality follows from an integration by parts. The above identity implies that $ \nabla R_T(\pw_T(\uv),\Pi^k_{\dt}(v|_{\dt})) = \nabla {\cal{E}}^{k+1}_T(v) $.
This  in \eqref{eq:4.18} gives 
\begin{align}
h_T^{-\frac{1}{2}}	\norm{S_{\dt}^{\rm{eo}}(\pw_T(\uv),\Pi^k_{\dt}(v|_{\dt}))}_{\dt}\lesssim h_T^{-1}\norm{\pw_T(\uv) - \Pi^k_T(v)}_{T} + \norm{\nabla(v-{\cal{E}}^{k+1}_T(v))}_T,\label{equal-stab-est}
\end{align}
which proves \eqref{stab-bound}.

\noindent (2) In the mixed-order case ($k'=k+1$), the definition \eqref{HHO-stab:1b} of $S_{\dt}^{\rm{mo}}$ and the $L^2$-stability of $\Pi^k_{\dt}$ lead to
\begin{align}
\norm{S_{\dt}^{\rm{mo}}(\pw_T(\uv),\Pi^k_{\dt}(v|_{\dt}))}_{\dt}  &= \norm{\Pi^k_{\dt}((\pw_T(\uv)-v)|_{\dt})}_{\dt}\nonumber
\\&\leq \norm{\pw_T(\uv)-v}_{\dt}\nonumber
\\&\leq\norm{\pw_T(\uv) - \Pi_T^{k'}(v)}_{\dt} +\norm{\Pi_T^{k'}(v)  - v}_{\dt},\label{4.16}
\end{align}
by the triangle inequality. The application of a discrete trace inequality for the first term and \eqref{3.16} for the second term on the right-hand side of \eqref{4.16} concludes the proof.
\end{proof}

\begin{lemma}[Relation between stabilization and divergence of dual variable]\label{lem:A3}
For all $T\in\cT$ and all $\uv:=(\bsg,v)\in (\bH^{\nu_{\bsg}}(T)\cap \bH(\nabla{\cdot};T))\times H^{\nu_v}(T)$ with $\nu_{\bsg}>\frac{1}{2}$ and $\nu_v>\frac{1}{2}$, we have
\begin{align*}
\tau_T(S_{\dt}^{\HHO}(\pw_T(\uv),\Pi^k_{\dt}(v|_{\dt})), S_{\dt}^{\HHO}(0,\Pi^k_{\dt}(w|_{\dt})))_{\dt} = (\nabla{\cdot} \bsg, w)_{T}\qquad\forall w\in\pbb^{k'}_\perp(T;\rbb).
\end{align*}
\end{lemma}
\begin{proof}
Let $w\in\pbb^{k'}_\perp(T;\rbb)$.	The equation \eqref{H:proj.c} with $\mu:=\Pi^k_{\dt}(w|_{\dt})$ and the $L^2$-orthogonality of $\Pi^k_{\dt}$  lead to
\begin{align*}
\tau_T(S_{\dt}^{\HHO}(\pw_T(\uv),\Pi^k_{\dt}(v|_{\dt})), {}&S_{\dt}^{\HHO}(0,\Pi^k_{\dt}(w|_{\dt})))_{\dt}\\ &= (\bsg-\ps_T(\uv)){\cdot}\bn_T,\Pi^k_{\dt}(w|_{\dt}))_{\dt} \\&= ((\bPi^k_{\dt}(\bsg|_{\dt})-\ps_T(\uv))){\cdot}\bn_T,w)_{\dt}\\&=((\bPi^k_{\dt}(\bsg|_{\dt})-\bsg){\cdot}\bn_T,w)_{\dt}+((\bsg-\ps_T(\uv))){\cdot}\bn_T,w)_{\dt}.
\end{align*}
The equation \eqref{H:proj.d} for the first term and an integration by parts for the second term imply that
\begin{align*}
\tau_T(S_{\dt}^{\HHO}(\pw_T(\uv),\Pi^k_{\dt}(v|_{\dt})),{} &S_{\dt}^{\HHO}(0,\Pi^k_{\dt}(w|_{\dt})))_{\dt}\\
&=(\ps_T(\uv)-\bsg, \nabla w)_{T}+ (\nabla{\cdot}(\bsg-\ps_T(\uv)),w)_{T} + (\bsg-\ps_T(\uv), \nabla w)_{T} \\&=(\nabla{\cdot}(\bsg-\ps_T(\uv)),w)_{T}\\&= (\nabla{\cdot}\bsg, w)_{T},
\end{align*}
where we used $(\nabla{\cdot} \ps_T(\uv), w)_{T} =0 $ since $w\in\pbb^{k'}_\perp(T;\rbb)$. This completes the proof.
\end{proof}

 Recall the polynomial space $\pbb^{k'}_\perp(T;\rbb)$ from \eqref{4.8}. For any linear form $\Phi:\pbb^{k'}_\perp(T;\rbb)\to \rbb$, we set $\norm{\Phi}_{(\pbb^{k'}_\perp)'}:=\sup_{w\in\pbb^{k'}_\perp(T;\rbb)\setminus\{0\}}\frac{\abs{\Phi(w)}}{\norm{w}_{T}}$.
\begin{lemma}[Bound on polynomials in $\pbb^{k'}_\perp(T;\rbb)$ by HHO stabilization]\label{lem:4.1}
For all $T\in\cT$ and all $p\in\pbb^{k'}_\perp(T;\rbb)$, we have
\begin{align*}
\norm{p}_{T}\lesssim h_T\norm{b_p}_{(\pbb^{k'}_\perp)'},
\end{align*}
where $b_p:\pbb^{k'}_\perp(T;\rbb)\to\rbb$ is the linear form defined by
\begin{align}
b_p(w):=(S_{\dt}^{\HHO}(p,0),S_{\dt}^{\HHO}(0,\Pi^k_{\dt}(w|_{\dt})))_{\dt}\qquad\forall w\in\pbb^{k'}_\perp(T;\rbb).\label{def:b}
\end{align}
\end{lemma}

\begin{proof}
Let $p\in\pbb_\perp^{k'}(T;\rbb)$. 
Invoking Lemma~\ref{lem:A1} followed by the lower bound from Lemma~\ref{prop:stability} with $\h{v}_T:=(p,0)\in\h{V}_T^k$ gives
\begin{align}
h_T^{-1}\norm{p}^2_{T}&\lesssim  h_T\norm{\nabla p}^2_T+\norm{p}^2_{\dt} \nonumber\\&\lesssim h_T\norm{\bG_T(p,0)}^2_{T}+\norm{S_{\dt}^{\HHO}(p,0)}^2_{\dt}.\label{eq:4.19}
\end{align}
Since $(\bG_T(p,0),\bq)_{T} = -(p,\nabla{\cdot}\bq)_{T} = 0$ for all $\bq\in\pbb^k(T;\rbb^d)$ because $\nabla{\cdot}\bq\in \pbb^{k-1}(T;\rbb)$, we infer that $\bG_T(p,0)=\bzero$.  Next, we observe that  $\norm{S_{\dt}^{\HHO}(p,0)}^2_{\dt}= -(S_{\dt}^{\HHO}(p,0), S_{\dt}^{\HHO}(0,\Pi^k_{\dt}(p|_{\dt})))_{\dt}$ owing to \eqref{stab-id}. The previous two identities in \eqref{eq:4.19} result in
\begin{align*}
h_T^{-1}\norm{p}^2_{T} \lesssim   \abs{(S_{\dt}^{\HHO}(p,0), S_{\dt}^{\HHO}(0,\Pi^k_{\dt}(p|_{\dt})))_{\dt}} = \abs{b_p(p)} \leq \norm{b_p}_{(\pbb^{k'}_\perp)'}\norm{p}_{T},
\end{align*}
where the last bound follows from the definition of the $\norm{\cdot}_{(\pbb^{k'}_\perp)'}$-norm. This concludes the proof.
\end{proof}

\section{Proofs on the H-interpolation operator} 
\label{sec:H-Int}
\subsection{Proof of well-posedness (Lemma~\ref{lem:5.1})}
The linear system defining the H-interpolation operator is square in both equal- and mixed-order cases. Indeed, the number of unknowns $(d+1)\text{dim}( \pbb^k(T;\mathbb{R}))$ (resp. $(d+1)\text{dim}( \pbb^{k}(T;\mathbb{R}))+\text{dim}(\tilde{\pbb}^{k+1}(T;\rbb))$), coincides with the number of equations, $(d+1)(\text{dim}(\pbb^{k-1}(T;\mathbb{R}))+\text{dim}(\pbb^{k}(F;\rbb)))$ (resp. $(d+1)(\text{dim}(\pbb^{k-1}(T;\mathbb{R}))+\text{dim}(\pbb^{k}(F;\rbb))) + \text{dim}(\tilde{\pbb}^{k+1}(T;\rbb))$). Hence, the system is uniquely solvable if and only if we prove that $\ul{\Pi}^{\sH}_T(\ul{0}) = \ul{0}:=(\bold{0},0)$. For $\ul{v} =\ul{0}$, the system \eqref{H:proj} reduces to
\begin{subequations}
\begin{alignat}{2}
(\Pi^v_T(\ul{0}),w)_T &= 0&\qquad &\forall w\in \pbb^{k-1}(T;\rbb),\label{H0.a}\\
(\bPi^{\bsg}_T(\ul{0}),\bxi)_T &= 0 &\qquad &\forall \bxi\in \pbb^{k-1}(T;\rbb^d),\label{H0.b} \\
-(\bPi^{\bsg}_T(\ul{0})\SCAL\bn_T,\mu)_{\dt} &= \tau_T (S_{\dT}^{\HHO}(\Pi^v_T(\ul{0}),0)),S_{\dt}^{\HHO}(0,\mu))_{\dT}  &\qquad&\forall \mu \in \pbb^k(\cF_{\dT};\rbb),\label{H0.c}\\
(\bPi^{\bsg}_T(\ul{0}),\nabla q)_T &= 0 &\qquad &\forall q \in \pbb^{k'}(T;\rbb). \label{H0.d}
\end{alignat}
\end{subequations}
\textit{Step 1 $(\Pi^k_{\dt}(\pw_T(\ul{0})|_{\dt}) = 0)$}. From \eqref{H0.a} with $w = \nabla{\cdot}\ps_T(\ul{0})$ and \eqref{H0.d} with $q=\pw_T(\ul{0})$, we have
\[(\Pi^v_T(\ul{0}), \nabla{\cdot}\ps_T(\ul{0}))_T = 0 = (\ps_T(\ul{0}),\nabla \pw_T(\ul{0}))_T.\]
This with an integration by parts shows that $(\ps_T(\ul{0}){\cdot}\bn_T,\pw_T(\ul{0}))_{\dt}=0$, which implies that $(\ps_T(\ul{0}){\cdot}\bn_T,$\\$\Pi^k_{\dt}(\pw_T(\ul{0})|_{\dt}))_{\dt}=0$ owing to the $L^2$-orthogonality of $\Pi^k_{\dt}$. Substituting this into \eqref{H0.c} with $\mu := \Pi^k_{\dt}(\pw_T(\ul{0})|_{\dt})$ and using $S_{\dT}^{\HHO}(\Pi^v_T(\ul{0}),0) = -S_{\dT}^{\HHO}(0,\Pi^k_{\dt}(\pw_T(\ul{0})|_{\dt}))$ (by \eqref{stab-id}  with $p:=\Pi^v_T(\ul{0})$) lead to
\begin{align} 
S_{\dT}^{\HHO}(\Pi^v_T(\ul{0}),0) = 0. \label{5.6}
\end{align}
In the equal-order case, the definition of the operator $R_T$ implies that, for all $q\in\pbb^{k+1}_*(T;\rbb)$, 
\begin{align*}
(\nabla R_T(\Pi^v_T(\ul{0}),0),\nabla q)_T = -(\Pi^v_T(\ul{0}),\Delta q)_T = 0,
\end{align*}
where the last identity follows from \eqref{H0.a} since $\Delta q\in \pbb^{k-1}(T;\rbb)$. Hence, $ R_T(\Pi^v_T(\ul{0}),0)\in\pbb^0(T;\rbb)$  so that  $(1-\Pi^k_T)R_T(\Pi^v_T(\ul{0}),0) = 0$. This in the definition \eqref{HHO-stab:1a} of $S_{\dt}^{\rm{eo}}$ shows that \[\Pi^k_{\dt}(\pw_T(\ul{0})|_{\dt})= S_{\dt}^{\rm{eo}}(\Pi^v_T(\ul{0}),0)  = 0,\] 
where the last equality follows from \eqref{5.6}. In the mixed-order case,  the definition \eqref{HHO-stab:1b} of $S_{\dt}^{\rm{mo}}$ directly implies that $\Pi^k_{\dt}(\pw_T(\ul{0})|_{\dt}) = 0$.

\noindent\textit{Step 2 $(\pw_T(\ul{0})=0)$}. Step~1 together with an integration by parts gives
\begin{align*}\norm{\nabla \pw_T(\ul{0})}^2_T &= -(\pw_T(\ul{0}),\Delta \pw_T(\ul{0}))_T + (\pw_T(\ul{0}),\nabla \pw_T(\ul{0}){\cdot}\bn_T)_{\dt}\\&=-(\pw_T(\ul{0}),\Delta \pw_T(\ul{0}))_T + (\Pi^k_{\dt}(\pw_T(\ul{0})|_{\dt}),\nabla \pw_T(\ul{0}){\cdot}\bn_T)_{\dt}=0,
\end{align*}
where we used \eqref{H0.a} since $\Delta \pw_T(\ul{0})\in\pbb^{k-1}(T;\rbb)$. Hence $ \pw_T(\ul{0})\in\pbb^0(T;\rbb)$, but we already know that $\Pi^k_{\dt}(\pw_T(\ul{0})|_{\dt}) = 0$. So we must have $\pw_T(\ul{0})=0$. 

\noindent\textit{Step 3 $(\ps_T(\ul{0})=\bold{0})$}. Recalling \eqref{5.6} and since $\bPi^{\bsg}_T(\ul{0})\SCAL\bn_T\in\pbb^k(\cF_{\dt};\rbb)$,  \eqref{H0.c} implies that $\bPi^{\bsg}_T(\ul{0})\SCAL\bn_T=0$. Since $\bPi^{k-1}_T(\bPi^{\bsg}_T(\ul{0})) =\bold{0}$ by \eqref{H0.b}, Lemma~\ref{lem:3.5} gives $\ps_T(\ul{0})=\bold{0}$. 

\subsection{Proof of approximation estimate (Theorem~\ref{prop:A1})}
{\bf (1) \textit{Estimate on $\norm{\pw_T(\uv) - \Pi^{k'}_T(v)}_T$}}. Set $e^v := \pw_T(\uv) - \Pi^{k'}_T(v)$. The equation \eqref{H:1a} and the definition of $\Pi^{k'}_T$ show that $e^v\in\pbb^{k'}_\perp(T;\rbb)$. Recalling the definition \eqref{def:b}, writing 
\[ (e^v,0) = (\pw_T(\uv),\Pi^k_{\dt}(v|_{\dt})) - (\Pi^{k'}_T(v),\Pi^k_{\dt}(v|_{\dt})),\] 
and using Lemma~\ref{lem:A3}, we infer that, for all $w\in\pbb^{k'}_\perp(T;\rbb)$,
\begin{align}
b_{e^v}(w)&=(S_{\dt}^{\HHO}(\pw_T(\uv),\Pi^k_{\dt}(v|_{\dt})),S_{\dt}^{\HHO}(0,\Pi^k_{\dt}(w|_{\dt})))_{\dt}\nonumber\\&\quad-(S_{\dt}^{\HHO}(\Pi^{k'}_T(v),\Pi^k_{\dt}(v|_{\dt})),S_{\dt}^{\HHO}(0,\Pi^k_{\dt}(w|_{\dt})))_{\dt}\nonumber\\&=\tau_T^{-1}(\nabla{\cdot}\bsg, w)_{T}-(S_{\dt}^{\HHO}(\Pi^{k'}_T(v),\Pi^k_{\dt}(v|_{\dt})),S_{\dt}^{\HHO}(0,\Pi^k_{\dt}(w|_{\dt})))_{\dt}\nonumber\\&  = \phi_{\bsg}(w) - \phi_v(w),\label{def:phi}
\end{align}
with the linear forms 
\begin{align*}
\phi_{\bsg}(w) &: = \tau_T^{-1}(\nabla{\cdot}\bsg, w)_{T}, \\
\phi_v(w)&:= (S_{\dt}^{\HHO}(\Pi^{k'}_T(v),\Pi^k_{\dt}(v|_{\dt})),S_{\dt}^{\HHO}(0,\Pi^k_{\dt}(w|_{\dt})))_{\dt}.
\end{align*} 
Since $e^v\in\pbb^{k'}_\perp(T;\rbb)$,  Lemma~\ref{lem:4.1}, and the triangle inequality show that
\begin{align}
\norm{e^v}_{T} \lesssim h_T\norm{b_{e^v}}_{(\pbb^{k'}_\perp)'} \lesssim    h_T(\norm{\phi_{\bsg}}_{(\pbb^{k'}_\perp)'}+\norm{\phi_v}_{(\pbb^{k'}_\perp)'}). \label{est:ev}
\end{align}
Since $w\in\pbb^{k'}_\perp(T;\rbb)$, we have $\phi_{\bsg}(w) =  \tau_T^{-1}(\nabla{\cdot}\bsg - \Pi^{k-1}_T(\nabla{\cdot}\bsg), w)_{T}$. Recalling that $\tau_T^{-1}=\ell_\Omega^{-1}h_T$, this proves that
\begin{align}
\norm{\phi_{\bsg}}_{(\pbb^{k'}_\perp)'}\leq \ell_\Omega^{-1}h_T\norm{\nabla{\cdot}\bsg - \Pi^{k-1}_T(\nabla{\cdot}\bsg)}_{T}.\label{est:bsg}
\end{align}
Furthermore, the Cauchy--Schwarz inequality, \eqref{S1} from Lemma~\ref{lem:3.4}, and Lemma~\ref{prop:stability} for $(0,w|_{\dt})\in \h{V}^k_{T}$ show that
\begin{align}
\abs{\phi_v(w)}&\leq \norm{S_{\dt}^{\HHO}(\Pi^{k'}_T(v),\Pi^k_{\dt}(v|_{\dt}))}_{\dt}\norm{S_{\dt}^{\HHO}(0,\Pi^k_{\dt}(w|_{\dt}))}_{\dt}\nonumber\\ &\lesssim h_T^{\frac{1}{2}}\norm{\nabla(v-{\cal{E}}_T^{k+1}(v))}_T \norm{\Pi^k_{\dt}(w|_{\dt})}_{\dt}\nonumber\\&\lesssim\norm{\nabla(v-{\cal{E}}_T^{k+1}(v))}_T \norm{w}_{T}\nonumber,
\end{align}
where we used the $L^2$-stability of $\Pi^k_{\dt}$ and a discrete trace inequality in the last step. This with the definition of the dual norm on $(\pbb^{k'}_\perp)'$ proves that
\begin{align}
\norm{\phi_v}_{(\pbb^{k'}_\perp)'}\lesssim \norm{\nabla(v-{\cal{E}}_T^{k+1}(v))}_T.\label{est:bv}
\end{align}	
The combination of \eqref{est:bsg} and \eqref{est:bv} in \eqref{est:ev} proves  the bound \eqref{est:pw}.

\noindent{\bf (2) \textit{Estimate on $\norm{\ps_T(\uv)-\bPi^k_T(\bsg)}_{T}$}}.
We fix $d$ faces $F_1,\dots,F_d$ of $T$. Applying Lemma~\ref{lem:3.5} to $\bq:= \bPi^k_T(\bsg)-\ps_T(\uv)$ and using \eqref{H:proj.b} which gives $\bPi^{k-1}_T(\bq)=\bold{0}$, we obtain	 
\begin{align}
\norm{\bPi^k_T(\bsg)-\ps_T(\uv)}_{T}
&\lesssim \sum_{i\in \{1:d\}}h_T^{\frac{1}{2}} \norm{(\bPi^k_T(\bsg)-\ps_T(\uv)){\cdot}\bn_{T}}_{F_i}\nonumber.
\end{align}
The discrete Cauchy--Schwarz inequality and triangle inequality imply that 
\begin{align}
\norm{\bPi^k_T(\bsg)-\ps_T(\uv)}_{T}
&\lesssim h_T^{\frac{1}{2}} \norm{(\bPi^k_T(\bsg)-\ps_T(\uv)){\cdot}\bn_{T}}_{\dt}\nonumber\\
& \lesssim h_T^{\frac{1}{2}} (\norm{(\bPi^k_T(\bsg)-\bPi^k_{\dt}(\bsg|_{\dt})){\cdot}\bn_{T}}_{\dt} +\norm{(\bPi^k_{\dt}(\bsg|_{\dt})-\ps_T(\uv)){\cdot}\bn_T}_{\dt})\nonumber\\&=:T_1+T_2.\label{4.27}
\end{align}
The $L^2$-stability of $\Pi^k_{\dt}$ gives
\begin{align}
T_1=h_T^{\frac{1}{2}}\norm{(\bPi^k_{\dt}((\bPi^k_T(\bsg)-\bsg)|_{\dt})){\cdot}\bn_T}_{\dt}
\leq h_T^{\frac{1}{2}}\norm{(\bPi^k_T(\bsg)-\bsg){\cdot}\bn_T}_{\dt}.\label{5.28}
\end{align}
Since $(\bPi^k_{\dt}(\bsg|_{\dt})-\ps_T(\uv)){\cdot}\bn_T \in \pbb^k(\cF_{\dt};\rbb)$, the $L^2$-orthogonality of $\Pi^k_{\dt}$, the definition \eqref{H:proj.c} of $\ps_T$ with   $\tau_T=\ell_\Omega h_T^{-1}$, see \eqref{def:tau}, and the Cauchy--Schwarz inequality  show that
\begin{align}
T_2&\lesssim   h_T^{\frac{1}{2}} \sup_{q\in\pbb^k(\cF_{\dt};\rbb)\setminus\{0\}}\frac{((\bPi^k_{\dt}(\bsg|_{\dt})-\ps_T(\uv)){\cdot}\bn_T, q)_{\dt}}{\norm{q}_{\dt}}\nonumber\\&
=h_T^{\frac{1}{2}} \sup_{q\in\pbb^k(\cF_{\dt};\rbb)\setminus\{0\}}\frac{((\bsg-\ps_T(\uv)){\cdot}\bn_T, q)_{\dt}}{\norm{q}_{\dt}}\nonumber\\&
=  h_T^{\frac{1}{2}} \sup_{q\in\pbb^k(\cF_{\dt};\rbb)\setminus\{0\}}\frac{\tau_T(S_{\dt}^{\HHO}(\pw_T(\uv),\Pi^k_{\dt}(v|_{\dt})), S_{\dt}^{\HHO}(0,q))_{\dt}}{\norm{q}_{\dt}}
\nonumber\\&\leq \sup_{q\in\pbb^k(\cF_{\dt};\rbb)\setminus\{0\}}\frac{\ell_\Omega h_T^{-\frac{1}{2}}\norm{S_{\dt}^{\HHO}(\pw_T(\uv),\Pi^k_{\dt}(v|_{\dt}))}_{\dt}\norm{S_{\dt}^{\HHO}(0,q)}_{\dt}}{\norm{q}_{\dt}}.\label{eq:4.20}
\end{align}
In the equal-order case, the definition \eqref{HHO-stab:1a} of $S_{\dt}^{\rm{eo}}$ implies that, for all $(0,q)\in\h{V}_T^k$,
\begin{align*}
\norm{S_{\dt}^{\rm{eo}}(0,q)}_{\dt}&\leq \norm{q}_{\dt}+\norm{(1-\Pi^k_T)R_T(0,q)}_{\dt}
\\&\lesssim \norm{q}_{\dt}+ h_T^{-\frac{1}{2}}\norm{(1-\Pi^k_T)R_T(0,q)}_{T}
\\&\lesssim \norm{q}_{\dt}+h_T^{\frac{1}{2}}\norm{\nabla R_T(0,q)}\lesssim \norm{q}_{\dt},
\end{align*}
using the triangle inequality, a discrete trace inequality, the Poincar\'e inequality in $T$, and the fact that $\norm{\nabla R_T(0,q)}_T\lesssim h_T^{-\frac{1}{2}}\norm{q}_{\dt}$ owing to the definition of $R_T$. 

In the mixed-order case, the definition \eqref{HHO-stab:1b} of $S_{\dt}^{\rm{mo}}$ directly implies that $\norm{S_{\dt}^{\rm{mo}}(0,q)}_{\dt}= \norm{q}_{\dt}$. Hence, altogether, we have $\norm{S_{\dt}^{\HHO}(0,q)}_{\dt}\lesssim \norm{q}_{\dt} $.  This in \eqref{eq:4.20} gives
\begin{align}
T_2&\lesssim   \ell_\Omega h_T^{-\frac{1}{2}}\norm{S_{\dt}^{\HHO}(\pw_T(\uv),\Pi^k_{\dt}(v|_{\dt}))}_{\dt}\nonumber\\&\lesssim \ell_\Omega h_T^{-1}\norm{\pw_T(\uv)- \Pi^{k'}_T(v)}_{T} + \ell_\Omega\norm{\nabla(v-{\cal{E}}^{k+1}_T(v))}_T\nonumber\\&\lesssim h_T\norm{\nabla{\cdot}\bsg - \Pi^{k-1}_T(\nabla{\cdot}\bsg)}_T+\ell_\Omega\norm{\nabla(v-{\cal{E}}^{k+1}_T(v))}_T,\label{4.21}
\end{align}
where we used  Lemma~\ref{lem4.4} and \eqref{est:pw}. Using the estimates \eqref{5.28} and  \eqref{4.21} in \eqref{4.27} concludes the proof of \eqref{est:ps}.

\subsection{Proof of approximation estimate with relaxed regularity (Theorem~\ref{thm:relax-reg})}

The weaker regularity assumption on $\nabla{\cdot}\bsg$ only affects the linear form $\phi_{\bsg}$ in \eqref{def:phi} which is now defined as  $\phi_{\bsg}(w): = \tau_T^{-1}\langle\nabla{\cdot}\bsg, w \rangle_{T}$   for all $w\in\pbb^{k'}_\perp(T;\rbb)$. Step~(1) in the proof of Theorem~\ref{prop:A1} is then modified as follows. Take $w\in\pbb^{k'}_\perp(T;\rbb)$.  Since $\langle\Pi^{k-1}_T(\nabla{\cdot}\bsg), w \rangle_T =(\Pi^{k-1}_T(\nabla{\cdot}\bsg), w)_T = 0$, we obtain
\begin{align*}
\ell_\Omega|\phi_{\bsg}(w)| = h_T |\langle\nabla{\cdot}\bsg, w \rangle_{T}|&= h_T |\langle\nabla{\cdot}\bsg-\Pi^{k-1}_T(\nabla{\cdot}\bsg), w \rangle_{T}|.
\end{align*}
The definition \eqref{def-dualnorm} of  the dual norm $\norm{\cdot}_{H^{s-1}(T)}$ and an inverse estimate for $w$ show that 
\begin{align*}
\ell_\Omega |\phi_{\bsg}(w)|& \leq h_T \|\nabla{\cdot}\bsg-\Pi^{k-1}_T(\nabla{\cdot}\bsg)\|_{H^{s-1}(T)}|w|_{H^{1-s}(T)}\\
& \lesssim h_T \norm{\nabla{\cdot}\bsg-\Pi^{k-1}_T(\nabla{\cdot}\bsg)}_{H^{s-1}(T)} h_T^{s-1} \norm{w}_{T} \\&= h_T^s \norm{\nabla{\cdot}\bsg-\Pi^{k-1}_T(\nabla{\cdot}\bsg)}_{H^{s-1}(T)}\norm{w}_{T}.
\end{align*}
This and  the definition of the $(\pbb^{k'}_\perp)'$-norm imply that
\begin{align}
\norm{\phi_{\bsg}}_{(\pbb^{k'}_\perp)'}\lesssim \ell_\Omega h_T^s \norm{\nabla{\cdot}\bsg-\Pi^{k-1}_T(\nabla{\cdot}\bsg)}_{H^{s-1}(T)}.
\end{align}
The rest of the proof of Theorem~\ref{prop:A1} remaining unchanged, this readily leads to the estimate \eqref{est:relax}.

\section{Proofs of error equations and error estimates}
\label{sec:errest}

\subsection{Error equations}\label{sec:error_eq}

\subsubsection{Proof of Lemma~\ref{lem:6.1}}
\textbf{(1)} \textit{Proof of \eqref{h-error:1a}}. The discrete problem \eqref{eq:HHO_flux} and the definition \eqref{def:G} of $\bG_\cT$ lead, for all $t\in\ol{J}$ and all $\bxi_\cT\in \bS_\cT^k$, to
\begin{align*}
&(\partial_t\bet^{\sH}_{\cT}(t),\bxi_\cT)_{\Omega} - (\bG_\cT(\hat{e}_\cM^{\sH}(t)),\bxi_\cT)_{\Omega}\nonumber\\
& =  -(\partial_t\bPi^{\bsg}_\cT(\ul{v}(t)),\bxi_\cT)_{\Omega}+ (\bG_\cT(\Pi^v_\cT(\uv(t)),\Pi_\cF^k(v(t)|_{\cF})),\bxi_\cT)_{\Omega}\nonumber\\
& =- (\partial_t\bPi^{\bsg}_\cT(\ul{v}(t)),\bxi_\cT)_{\Omega}+\sum_{T\in\cT} \Big\{- (\Pi^v_T(\ul{v}(t)), \nabla{\cdot}\bxi_T)_T + (\Pi^k_{\partial T}(v(t)|_{\dt}),\bxi_T{\cdot}\bn_T)_{\partial T}\Big\}\nonumber\\& 
=- (\partial_t\bPi^{\bsg}_\cT(\ul{v}(t)),\bxi_T)_{\Omega} +\sum_{T\in\cT}\big\{ - (v(t),\nabla{\cdot}\bxi_T)_T + (v(t),\bxi_T{\cdot}\bn_T)_{\partial T}\big\}\nonumber\\& = - (\partial_t\bPi^{\bsg}_\cT(\ul{v}(t)),\bxi_\cT)_{\Omega} +(\nabla v(t),\bxi_T)_{\Omega} = (\partial_t(\bsg(t)-\bPi^{\bsg}_\cT(\ul{v}(t))),\bxi_\cT)_{\Omega},
\end{align*}
where we used the definition  of $\Pi^v_T$, \eqref{H:proj.a}, and the $L^2$-orthogonality of $\Pi^k_{\dt}$ on the fourth line (observe that $\nabla{\cdot}\bxi_T \in \pbb^{k-1}(T;\rbb)$ and $\bxi_T{\cdot}\bn_T\in\pbb^k(\cF_{\dt};\rbb)$),  an integration by parts and the model problem \eqref{eq:model.1a} on the fifth line. This proves \eqref{h-error:1a}.

\noindent\textbf{(2)} \textit{Proof of \eqref{h-error:1b}}.	The discrete problem \eqref{eq:HHO_primal} and the continuous problem \eqref{eq:model.1b} show that, for all $t\in\ol{J}$ and all $\h{w}_\cM\in\h{V}_{\cM 0}^k$,
\begin{align}
&(\partial_te^{\sH}_\cT(t),w_\cT)_{\Omega}+(\bet^{\sH}_\cT(t),\bG_\cT(\h{w}_\cM))_{\Omega} +s_{\cM}(\hat{e}_\cM^{\sH}(t),\hat{w}_\cM)\nonumber\\& = ( \partial_tv(t)-\nabla{\cdot}\bsg(t),w_\cT)_{\Omega} - (\partial_t \Pi^v_\cT(\ul{v}(t)),w_\cT)_{\Omega}-(\bPi^{\bsg}_\cT(\ul{v}(t)),\bG_\cT(\h{w}_\cM))_{\Omega}\nonumber\\&\qquad-s_{\cM}((\Pi^v_\cT(\uv(t)),\Pi_\cF^k(v(t)|_\cF)),\hat{w}_\cM)\nonumber\\& =(\partial_t(v(t)-\Pi^v_\cT(\ul{v}(t))),w_\cT)_{\Omega}+ \sum_{T\in\cT}\Big\{(\bsg(t),\nabla w_T)_{T}- (\bsg(t){\cdot}\bn_T , w_T)_{\partial T}-(\bPi^{\bsg}_T(\ul{v}(t)), \nabla w_T)_{T} \nonumber \\&\qquad  + (\bPi^{\bsg}_T(\ul{v}(t)){\cdot}\bn_T,w_T-w_{\partial T})_{\partial T}\Big\} -  s_{\cM}((\Pi^v_\cT(\uv(t)),\Pi_\cF^k(v(t)|_\cF)),\hat{w}_\cM),\nonumber
\end{align}
where we used an integration by parts (recall that $\boldsymbol{\sigma}\in\boldsymbol{H}^{\nu_{\boldsymbol{\sigma}}}(\Omega), \nu_{\boldsymbol{\sigma}}>\frac12$ by assumption), and the definition  \eqref{def:G} of $\bG_\cT$ on the third line. Noticing that 
\begin{align*}
(\bsg(t),\nabla w_T)_{T}- (\bsg(t){\cdot}\bn_T , w_T)_{\partial T}-(\bPi^{\bsg}_T(\ul{v}(t)), \nabla w_T)_{T} = -(\bPi^k_{\dt}(\bsg(t)|_{\dt}){\cdot}\bn_T,w_T)_{\dt}
\end{align*}
owing to \eqref{H:proj.d},  we arrive at
\begin{align}
&(\partial_te^{\sH}_\cT(t),w_\cT)_{\Omega}+(\bet^{\sH}_\cT(t),\bG_\cT(\h{w}_\cM))_{\Omega} +s_{\cM}(\hat{e}_\cM^{\sH}(t),\hat{w}_\cM)\nonumber\\&
=(\partial_t(v(t)-\Pi^v_\cT(\ul{v}(t))),w_\cT)_{\Omega} +\sum_{T\in\cT}\Big\{((\bPi^{\bsg}_T(\ul{v}(t))-\bPi^k_{\partial T}(\bsg(t)|_{\dt}){\cdot}\bn_T, w_T-w_{\partial T})_{\partial T}\Big\}\nonumber\\&\qquad
-  s_{\cM}((\Pi^v_\cT(\uv(t)),\Pi_\cF^k(v(t)|_\cF)),\hat{w}_\cM),\label{5.5}
\end{align}
where we used that  $\sum_{T\in\cT}(\bPi^k_{\dt}(\bsg(t)|_{\dt}){\cdot}\bn_T,w_{\dt})_{\dt}=0$. Invoking the $L^2$-orthogonality of $\Pi^k_{\dt}$, using \eqref{H:proj.c} with $\mu:=\Pi^k_{\partial T}(w_T|_{\dt})-w_{\partial T}$, and observing that $S_{\dt}^{\HHO}(\h{w}_T) = -S_{\dt}^{\HHO}(0,\Pi^k_{\partial T}(w_T|_{\dt})-w_{\partial T}))_{\dt}$ owing to \eqref{stab-id}, we infer that
\begin{align*}
&\sum_{T\in\cT}\Big\{((\bPi^{\bsg}_T(\ul{v}(t))-\bPi^k_{\partial T}(\bsg(t)|_{\dt}){\cdot}\bn_T, w_T-w_{\partial T})_{\partial T}\Big\}\nonumber\\
&\quad=\sum_{T\in\cT}((\bPi^{\bsg}_T(\ul{v}(t))-\bsg(t)){\cdot}\bn_T, \Pi^k_{\partial T}(w_T|_{\dt})-w_{\partial T})_{\partial T}\nonumber\\&\quad
=-\sum_{T\in\cT}\tau_T(S_{\dt}^{\HHO}(\Pi^v_T(\uv(t)),\Pi_{\dt}^k(v(t)|_{\dt})),S_{\dt}^{\HHO}(0,\Pi^k_{\partial T}(w_T|_{\dt})-w_{\partial T}))_{\dt}\\
&\quad =s_{\cM}((\Pi^v_\cT(\uv(t)),\Pi_\cF^k(v(t)|_\cF)),\hat{w}_\cM).
\end{align*}
This  in \eqref{5.5} concludes the proof of \eqref{h-error:1b}.

\subsubsection{Proof of Lemma~\ref{lem:err_HHO}}

One can refer to \cite[Theorem 3]{BDES_2021}.

\subsubsection{Proof of Corollary~\ref{cor:comb_err}}

Shifting the terms on the right-hand side of \eqref{h-error} to the left-hand side gives
\begin{align*}
&\Big\{(\pt\bet_\cT^{\sH}(t),\bxi_\cT)_{\Omega}-(\pt(\bsg(t)-\bPi_\cT^{\bsg}(\uv(t))),\bxi_\cT)_{\Omega}\Big\} - (\bG_\cT(\h{e}_\cM^{\sH}(t)),\bxi_\cT)_{\Omega}=0,\\
&\Big\{(\pt e_\cT^{\sH}(t),w_\cT)_{\Omega}-(\partial_t(v(t)-\Pi^v_\cT(\ul{v}(t))),w_\cT)_{\Omega}\Big\} +(\bet_\cT^{\sH}(t), \bG_\cT(\h{w}_\cM))_{\Omega} +s_{\cM}(\h{e}_\cM^{\sH}(t),\h{w}_\cM) = 0.
\end{align*}
The $L^2$-orthogonality of  $\bPi^k_{\cT}$ and $\Pi^{k'}_{\cT}$ lead to
\begin{align*}
(\pt\bet_\cT^{\sH}(t),\bxi_\cT)_{\Omega}-(\pt(\bsg(t)-\bPi_\cT^{\bsg}(\uv(t))),\bxi_\cT)_{\Omega} &= (\pt\bet_\cT^{\sPi}(t),\bxi_\cT)_{\Omega},\\
(\pt e_\cT^{\sH}(t),w_\cT)_{\Omega}-(\partial_t(v(t)-\Pi^v_\cT(\ul{v}(t)),w_\cT)_{\Omega} &= (\pt e_\cT^{\sPi}(t),w_\cT)_{\Omega},
\end{align*}
and, consequently, prove \eqref{h:error-new}.

\subsubsection{Proof of Corollary~\ref{cor:init-dt}}

From the error equation \eqref{h-error:1b} at $t=0$ and the initial conditions $\bet^{\sH}_\cT(0) = \boldsymbol{0}$ and $e^{\sH}_\cT(0) =0$, we infer that \[s_{\cM}((0,e_\cF^{\sH}(0)),(0,w_\cF)) = 0\qquad \forall w_\cF\in V_{\cF 0}^k.\] This implies that  $e_\cF^{\sH}(0) = 0$. Invoking these initial conditions in \eqref{h:error-new.a}-\eqref{h:error-new.b} at $t=0$ proves the claim.

\subsection{The energy argument (Proof of Theorem~\ref{thm:energy-error})}
\label{sec:analysis}

For all $t\in \ol{J}$, recall the space semi-discrete errors
\begin{alignat*}{2}
\bet_\cT^{\sH} (t) &:= \bsg_\cT(t) - \bPi^{\bsg}_\cT(\ul{v}(t)), \qquad &\h{e}_\cM^{\sH}(t) &:= \h{v}_\cM(t)-(\Pi^v_\cT(\uv(t)),\Pi_\cF^k(v(t)|_\cF)),\\
\bet_\cT^{\sPi} (t) &:= \bsg_\cT(t) - \bPi^{k}_\cT(\bsg(t)), \qquad &\h{e}_\cM^{\sPi}(t) &:= \h{v}_\cM(t)-(\Pi^{k'}_\cT(v(t)),\Pi^k_\cF(v(t)|_\cF)).
\end{alignat*}

{\bf (1)} As shown in \cite{BDES_2021}, substituting $\h{w}_\cM := \h{e}_\cM^{\sPi}(t)$ in \eqref{hho-error:1b} and $\bxi_\cT := \bet_\cT^{\sPi}(t)$ in \eqref{hho-error:1a} for all $t\in \ti{J}$, we obtain after some straightforward manipulations
\begin{align*}
&\norm{e_\cT^{\sPi}}^2_{C^0(J_t;\Omega)} + \norm{\bet_\cT^{\sPi}}^2_{C^0(J_t;\Omega)} + |\h{e}_\cM^{\sPi}|^2_{L^2(J_t;\sS)}\lesssim \norm{e_\cT^{\sPi}(0)}^2_\Omega+ \norm{\bet_\cT^{\sPi}(0)}^2_\Omega \nonumber\\
&\hspace{5cm}+\norm{\psi_{\cM}^{\sPi}(\uv;\cdot)}^2_{L^{\infty}(J_t;(\sPi)')}+\norm{\psi_{\cM}^{\sPi}(\pt \uv;\cdot)}^2_{L^{1}(J_t;(\sPi)')}.
\end{align*}
The initial condition \eqref{def:initc} and the  Pythagoras identity based on the $L^2$-orthogonalities of $\Pi_\cT^{k'}$ and $\bPi_\cT^k$ imply that
\begin{align*}
	\norm{e_\cT^{\sPi}(0)}_\Omega &= \norm{\Pi_\cT^v(\ul{v}_0)-\Pi_\cT^{k'}(v_0)}_\Omega \leq \norm{v_0-\Pi_\cT^v(\ul{v}_0)}_\Omega,\\
	\norm{\bet_\cT^{\sPi}(0)}_\Omega &= \norm{\bPi_\cT^{\bsg}(\ul{v}_0)-\bPi_\cT^{k}(\bsg_0)}_\Omega \leq \norm{\bsg_0-\bPi_\cT^{\bsg}(\ul{v}_0)}_\Omega.
\end{align*}
Inserting these inequalities in the first inequality of this step and taking the square root proves \eqref{energy-err.a}.

{\bf (2)} Differentiating in time the error equations, we can similarly obtain
\begin{align*}
&\norm{\partial_te^{\sPi}_\cT}^2_{C^0(J_t;\Omega)}+\norm{\partial_t\bet^{\sPi}_\cT}^2_{C^0(J_t;\Omega)} + |\partial_t\h{e}^{\sPi}_\cM|^2_{L^2(J_t;\sS)} \lesssim \norm{\partial_te^{\sPi}_\cT(0)}^2_{\Omega}  +\norm{\partial_t\bet^{\sPi}_\cT(0)}^2_{\Omega}\nonumber\\
&\hspace{5cm}+
\norm{\psi_{\cM}^{\sPi}(\pt \uv;\cdot)}_{L^{\infty}(J_t;(\sPi)')}^{2}+\norm{\psi_{\cM}^{\sPi}(\partial_{tt} \uv;\cdot)}_{L^{1}(J_t;(\sPi)')}^{2}.
\end{align*}	
This followed by the zero initial conditions \eqref{IC:time-diff}  concludes the proof of \eqref{energy-err.b}.

\subsection{The duality argument (Proof of Theorem~\ref{thm:tim-int-err})}

{\bf Step 1 (Dual problems)} Following \cite{cockburn2014uniform}, we introduce a dual problem and a time-integrated dual problem. However, we only assume that the elliptic regularity pickup satisfies  $s\in(\frac{1}{2},1]$ instead of $s=1$ in \cite{cockburn2014uniform}.  Let $\theta\in L^2(\Omega)$, and let $\Phith$ be the solution of the homogeneous wave equation 
\begin{align}
\partial_{tt}\Phith-\Delta\Phith = 0 \qquad\text{in}\; J\times\Omega,\label{phi-dual}
\end{align}
with the \emph{final} and boundary conditions
\begin{subequations}
\begin{align}
	\Phith(T_f) = 0,\;\pt\Phith(T_f) = \theta\qquad&\text{on}\;\Omega,\label{Final-Phi}\\
	\Phith(t,x) = 0\qquad&\text{on}\; J\times\Gamma.
\end{align}
\end{subequations}
The following regularity result is classical (see, e.g., \cite[Chapter~8]{allaire2007numerical}):
\begin{align}
\norm{\nabla \Phith}_{L^\infty(J;\bL^2(\Omega))}+\norm{\pt\Phith}_{L^\infty(J;L^2(\Omega))} \lesssim \norm{\theta}_{\Omega}.\label{phi-reg}
\end{align}
We set $\tphith(t):=\int_t^{T_f}\Phith(s)\,\ds$ for all $t\in \ol{J}$. A straightforward calculation shows that $\tphith$ satisfies the inhomogenous wave equation
\begin{align}
\partial_{tt}\tphith-\Delta\tphith = -\theta \quad\text{in}\; J\times\Omega,\label{tphi-dual}
\end{align}
with the final and boundary conditions
\begin{subequations}
\begin{align}
	\tphith(T_f) = 0,\;\pt\tphith(T_f) = 0\quad&\text{on}\;\Omega,\\
	\tphith(t,x) = 0\quad&\text{on}\; J\times\Gamma.
\end{align}
\end{subequations}
Moreover, invoking elliptic regularity in $\Omega$, we infer that there exists some $s\in(\frac{1}{2},1]$ such that 
\begin{align}
\norm{\tphith}_{L^{\infty}(J;L^2(\Omega))} +
\ell_\Omega\norm{\nabla \tphith}_{L^{\infty}(J;\boldsymbol{L}^2(\Omega))}+\, 
\ell_\Omega^{1+s}\abs{\tphith}_{L^{\infty}(J;H^{1+s}(\Omega))} \lesssim \ell_\Omega ^2\norm{\theta}_{\Omega}.\label{tphi-reg}
\end{align}

\noindent{\bf Step 2 (Key identity)}. Let $(\theta_n)_{n\in\mathbb{N}}\subset (C_0^{\infty}(\Omega))^{\mathbb{N}}$ be a sequence such that $\theta_n\to \theta$ in $L^2(\Omega)$. Without loss of generality, we assume that $\norm{\theta_n}_\Omega\lesssim 2\norm{\theta}_\Omega$ for all $n\in\mathbb{N}$. Observe that $\Phi_{\theta_n}$ is smooth in space and in time. The final boundary condition $\pt\Phithn(T_f) = \theta_n$ and  $\ut{e}^{\sPi}_\cT(0)=0$ lead to
\begin{align}
\fE(\Phithn):=(\ut{e}_\cT^{\sPi}(T_f),\theta_n)_{\Omega} &= (\ut{e}_\cT^{\sPi}(T_f),\pt\Phithn(T_f))_{\Omega}\nonumber\\
&=\int_J\frac{d}{dt}(\ut{e}^{\sPi}_\cT(t),\pt\Phithn(t))_{\Omega}\; dt\nonumber\\&=\int_J\left\{(e_\cT^{\sPi}(t),\pt\Phithn(t))_{\Omega}+(\Delta\Phithn(t),\ut{e}^{\sPi}_\cT(t))_{\Omega}\right\}\,dt,\label{6.18}
\end{align}
where we used  $\pt\ut{e}^{\HHO}_\cT(t)=e_\cT^{\HHO}(t)$ and $\partial_{tt}\Phithn = \Delta\Phithn$ in the last step. For all $f\in H^1(J;L^2(\Omega))$ and all $g\in L^2(J;L^2(\Omega))$ with $\ti{g}(t):=\int_t^{T_f}g(s)\,\ds$, the following identity holds: 
\begin{align}
	\int_J(f(t),g(t))_\Omega\,dt =(f(0), \ti{g}(0))_\Omega+\int_J(\pt f(t),\ti{g}(t))_\Omega\,dt.\label{time-int}
\end{align}  
The identity \eqref{time-int} for $f:=\ut{e}^{\sPi}_\cT$ (so that $f(0)=0$ and $\pt f =e^{\sPi}_\cT$) and $g:=\Delta \Phi_{\theta_n}$ implies that
\begin{align}
	\int_J(\Delta \Phi_{\theta_n}(t),\ut{e}^{\sPi}_\cT(t))_\Omega\,dt &= \int_J(\Delta \ti{\Phi}_{\theta_n}(t),e^{\sPi}_\cT(t))_\Omega\,dt\nonumber\\&= -\int_J\sum_{T\in\cT}\Big\{(\nabla \ti{\Phi}_{\theta_n}(t) , \nabla e^{\sPi}_\cT(t))_T-(\nabla \ti{\Phi}_{\theta_n}(t){\cdot}\bn_T,e^{\sPi}_\cT(t))_{\dt}\Big\}dt,\label{7.15}
\end{align}
where we used cellwise integration by parts in space on the second line.
This in \eqref{6.18} leads to
\begin{align*}
\fE(\Phithn)
&=\int_J\Big\{(e_\cT^{\sPi}(t),\pt\Phithn(t))_{\Omega}-\sum_{T\in\cT}\big\{(\nabla \ti{\Phi}_{\theta_n}(t) , \nabla e^{\sPi}_\cT(t))_T-(\nabla \ti{\Phi}_{\theta_n}(t){\cdot}\bn_T,e^{\sPi}_\cT(t))_{\dt}\big\}dt\Big\}.
\end{align*}

\noindent We now pass to the limit $n\to\infty$ in the above identity.
From \eqref{phi-reg}, \eqref{tphi-reg}, and the linearity of the wave equation, we have $\pt\Phithn\to\pt\Phith$ in $L^\infty(J;L^2(\Omega))$ and  $\ti{\Phi}_{\theta_n}\to \tphith$ in $L^\infty(J;H^{1+s}(\Omega))$.
Since we have that

   \[
\norm{\pt\Phithn}_{L^\infty(J;L^2(\Omega))}
+\ell_\Omega^{-1}\norm{\nabla \ti{\Phi}_{\theta_n}}_{L^{\infty}(J;\boldsymbol{L}^2(\Omega))}
+\ell_\Omega^{-2}\norm{\ti{\Phi}_{\theta_n}}_{L^{\infty}(J;H^{1+s}(\Omega))}\lesssim \norm{\theta_n}_\Omega \lesssim 2\norm{\theta}_\Omega, \]
for all $n\mathbb{N}$,
we can invoke Lebesgue's dominated convergence theorem and infer that
\begin{align}
\fE(\Phith)&:=(\ut{e}_\cT^{\sPi}(T_f),\theta)_{\Omega}\nonumber \\
&=\int_J\Big\{(e_\cT^{\sPi}(t),\pt\Phith(t))_{\Omega}-\sum_{T\in\cT}\big\{(\nabla \ti{\Phi}_{\theta}(t) , \nabla e^{\sPi}_\cT(t))_T-(\nabla \ti{\Phi}_{\theta}(t){\cdot}\bn_T,e^{\sPi}_\cT(t))_{\dt}\big\}dt\Big\}\nonumber\\&=:T_1+T_2.\label{6.19}
\end{align}
On the one hand, invoking \eqref{time-int} with $f:=e_\cT^{\sPi}\in H^1(J;L^2(\Omega))$ and $g:=\pt\Phith\in L^2(J;L^2(\Omega))$ and since  $\ti{g}(t)= \Phith(T_f)- \Phith(t)=-\Phith(t)$ owing to \eqref{Final-Phi},  we infer that 
\begin{align}
T_1 = -(e_\cT^{\sPi}(0),\Phith(0))_\Omega-\int_J(\pt e^{\sPi}_\cT(t),\Phith(t))_{\Omega}\,dt.\label{6.20}
\end{align}
On the other hand, the $L^2$-orthogonality of $\bPi^k_{\cT}$ leads to
\begin{align}
T_2 = -\int_J \sum_{T\in\cT}\Big\{(\bPi^k_{\cT}(\nabla\tphith(t)),\nabla e^{\sPi}_T(t))_{T}-(\nabla\tphith(t){\cdot}\bn_T,e^{\sPi}_T(t))_{\partial T}\Big\}\,dt.\label{eq:4.29}
\end{align}
The definition of $\bG_\cT$ in the first error equation \eqref{hho-error:1a} implies that, for all $\bxi_\cT\in\bS^k_\cT$,
\begin{align*}
-(\nabla e_\cT^{\sPi}(t),\bxi_\cT)_{\Omega} &
=-(\pt\bet_\cT^{\sPi}(t),\bxi_\cT)_{\Omega}+ \sum_{T\in\cT}(e_{\dt}^{\sPi}(t)-e_T^{\sPi}(t),\bxi_T{\cdot}\bn_T)_{\dt}.
\end{align*}
This with $\bxi_\cT := \bPi^k_\cT(\nabla\tphith(t))$  and the observation that $\sum_{T\in\cT} (e^{\sPi}_{\dt}(t),\nabla\tphith(t){\cdot}\bn_T))_{\dt}=0$ in \eqref{eq:4.29} (owing to the regularity in space of $\tphith$)  show that
\begin{align}
T_2 &=-\int_J\Big\{(\pt\bet_\cT^{\sPi}(t),\bPi^{k}_\cT(\nabla\tphith(t)))_{\Omega}-\sum_{T\in\cT}(e_T^{\sPi}(t)-e_{\dt}^{\sPi}(t),(\nabla\tphith(t)-\bPi^{k}_T(\nabla\tphith(t))){\cdot}\bn_T)_{\dt}\Big\}dt.\label{eq:4.30}
\end{align}
Differentiating in time the second error equation \eqref{hho-error:1b}, we obtain, for all $\h{w}_\cM\in\h{V}_{\cM 0}^k$,
\begin{align*}
- (\pt\bet_\cT^{\sPi}(t), \bG_\cT(\h{w}_\cM))_{\Omega}=(\partial_{tt} e_\cT^{\sPi}(t),w_\cT)_{\Omega}  +s_{\cM}(\pt\h{e}_\cM^{\sPi}(t),\h{w}_\cM)- \psi_{\cM}^{\sPi}(\pt \uv(t);\h{w}_\cM).
\end{align*}
This with $\h{w}_\cM :=(\Pi_\cT^{k'}(\tphith(t)),\Pi^k_{\cF}(\tphith(t)|_{\cF}))$ and the observation that $\bG_\cT(\Pi_\cT^{k'}(\tphith(t)),\Pi^k_{\cF}(\tphith(t)|_{\cF})) = \bPi^k_\cT(\nabla\tphith(t))$ in \eqref{eq:4.30} lead to
\begin{align}
T_2 &= \int_J\Big\{  (\partial_{tt} e_\cT^{\sPi}(t), \Pi_\cT^{k'}(\tphith(t)))_{\Omega} +s_{\cM}(\pt\h{e}_{\cM}^{\sPi}(t),(\Pi_\cT^{k'}(\tphith(t)),\Pi^k_{\cF}(\tphith(t)|_{\cF})))\nonumber\\&\quad - \psi_{\cM}^{\sPi}(\pt \uv(t);(\Pi_\cT^{k'}(\tphith(t)),\Pi^k_{\cF}(\tphith(t)|_{\cF})))\nonumber\\&
\quad+\sum_{T\in\cT}(e_T^{\sPi}(t)-e_{\dt}^{\sPi}(t),(\nabla\tphith(t)-\bPi^{k}_T(\nabla\tphith(t))){\cdot}\bn_T)_{\dt}\Big\}\,dt\label{eq:4.31}.
\end{align}
Since $\tphith(T_f) =0$,  $\pt\tphith(t) = -\Phith(t)$, and $\pt e_\cT^{\sPi}(0)=0$ owing to \eqref{IC:time-diff},  we have
\begin{align*}
\int_J(\partial_{tt} e_\cT^{\sPi}(t), \Pi_\cT^{k'}(\tphith(t)))_{\Omega}\,dt &=\int_J\Big\{\frac{d}{dt}(\pt e_\cT^{\sPi}(t),\Pi_\cT^{k'}(\tphith(t)))_{\Omega}-(\pt e_\cT^{\sPi}(t),\Pi_\cT^{k'}(\pt\tphith(t)))_{\Omega}\Big\}\,dt\\&=  \int_J(\pt e_\cT^{\sPi}(t),\Pi^{k'}_\cT(\Phith(t)))_{\Omega}\,dt.
  \end{align*}
This with the $L^2$-orthogonality of $\Pi^{k'}_{\cT}$ in \eqref{eq:4.31}  shows that
\begin{align}
	T_2 &= \int_J\Big\{(\pt e_\cT^{\sPi}(t),\Phith(t))_{\Omega}\,dt+s_{\cM}(\pt\h{e}_{\cM}^{\sPi}(t),(\Pi_\cT^{k'}(\tphith(t)),\Pi^k_{\cF}(\tphith(t)|_{\cF})))\nonumber\\&\quad  - \psi_{\cM}^{\sPi}(\pt \uv(t);(\Pi_\cT^{k'}(\tphith(t)),\Pi^k_{\cF}(\tphith(t)|_{\cF})))\nonumber\\&
	\quad+\sum_{T\in\cT}(e_T^{\sPi}(t)-e_{\dt}^{\sPi}(t),(\nabla\tphith(t)-\bPi^{k}_T(\nabla\tphith(t))){\cdot}\bn_T)_{\dt}\Big\}\,dt\label{eq:4.32}.
\end{align}
The combination of \eqref{6.20} and \eqref{eq:4.32} in \eqref{6.19} proves that
\begin{align}
\fE(\Phith) &= -( e_\cT^{\sPi}(0),\Phith(0))_{\Omega} + \int_J\Big\{s_{\cM}(\pt\h{e}_{\cM}^{\sPi}(t),(\Pi_\cT^{k'}(\tphith(t)),\Pi^k_{\cF}(\tphith(t)|_{\cF})))\nonumber\\&\quad- \psi_{\cM}^{\sPi}(\pt \uv(t);(\Pi_\cT^{k'}(\tphith(t)),\Pi^k_{\cF}(\tphith(t)|_{\cF}))) \nonumber\\&\quad +\sum_{T\in\cT}(e_T^{\sPi}(t)-e_{\dt}^{\sPi}(t),(\nabla\tphith(t)-\bPi^{k}_T(\nabla\tphith(t))){\cdot}\bn_T)_{\dt}\Big\}\,dt
\nonumber\\&=:\fE_0(\Phith(0))+\int_J\sum_{i=1}^3\fE_i(\tphith(t))\,dt.\label{5.26}
\end{align}
{\bf Step 3 (Bounds on $\{\fE_j\}_{j=0}^3$)}. For $\fE_0$, the observation $e_\cT^{\sPi}(0) = e_\cT^{\sH}(0) + \Pi^v_\cT(\ul{v}_0)-\Pi_\cT^{k'}(v_0)  = \Pi^v_\cT(\ul{v}_0)-\Pi_\cT^{k'}(v_0)$ owing to the initial condition $e_\cT^{\sH}(0) = 0 $ and the Cauchy--Schwarz inequality show that 
\begin{align}
|\fE_0(\Phith(0))| 
&\leq \norm{\Pi_\cT^{k'}(v_0)-\Pi^v_\cT(\ul{v}_0)}_{\Omega} \norm{\Phith(0)}_{\Omega}.\label{E2}
\end{align}	 
For $\fE_1$, the Cauchy--Schwarz inequality followed by \eqref{S1} and the classical approximation properties for the elliptic projection ${\cal E}_{\cT}^{k+1}$ give
\begin{align}
|\fE_1(\tphith(t))| &\leq |\pt\h{e}_{\cM}^{\sPi}(t)|_{\sS} |(\Pi_\cT^{k'}(\tphith(t)),\Pi^k_{\cF}(\tphith(t)|_{\cF}))|_{\sS}\lesssim \ell_\Omega^{\frac12}
h^s |\pt\h{e}_{\cM}^{\sPi}(t)|_{\sS} |\tphith(t)|_{H^{1+s}(\Omega)}.\label{E4}
\end{align}
For $\fE_2$, substituting $\h{w}_{\cM}:= (\Pi_\cT^{k'}(\tphith(t)),\Pi^k_{\cF}(\tphith(t)|_{\cF}))$ in \eqref{cons-HHO} and since $\bG_\cT(\h{w}_\cM) = \bPi^k_\cT(\nabla\tphith(t))$, we infer that
\begin{align}
\fE_2(\tphith(t)) &= (\nabla{\cdot}\pt\bsg(t),\Pi_\cT^{k'}(\tphith(t)))_\Omega + (\bPi^k_{\cT}(\pt\bsg(t)),\bPi^k_\cT(\nabla\tphith(t)))_\Omega\nonumber\\&\quad+ s_{\cM}((\Pi^{k'}_\cT(\pt v(t)),\Pi^k_\cF(\pt v(t)|_\cF)),(\Pi_\cT^{k'}(\tphith(t)),\Pi^k_{\cF}(\tphith(t)|_{\cF}))).\label{7.27}
\end{align}
The $L^2$-orthogonality of the projection operators and the fact that $(\nabla{\cdot}\pt\bsg(t),\tphith(t))_\Omega = -(\pt\bsg(t),\nabla\tphith(t))_\Omega$ owing to the boundary condition $\tphith(t)=0$ on $\Gamma$ show that
\begin{align*}
&(\nabla{\cdot}\pt\bsg(t),\Pi_\cT^{k'}(\tphith(t)))_\Omega + (\bPi^k_{\cT}(\pt\bsg(t)),\bPi^k_\cT(\nabla\tphith(t)))_\Omega\nonumber\\& =(\nabla{\cdot}\pt\bsg(t),\Pi_\cT^{k'}(\tphith(t))-\tphith(t))_\Omega +(\nabla{\cdot}\pt\bsg(t),\tphith(t))_\Omega+ (\pt\bsg(t),\bPi^k_\cT(\nabla\tphith(t)))_\Omega\\
& = (\nabla{\cdot}\pt\bsg(t)-\Pi_\cT^{k'}(\nabla{\cdot}\pt\bsg(t)),\Pi_\cT^{k'}(\tphith(t))-\tphith(t))_\Omega - (\pt\bsg(t)-\bPi^k_{\cT}(\pt\bsg(t)),\nabla\tphith(t)-\bPi^k_\cT(\nabla\tphith(t)))_\Omega.
\end{align*}
This in the identity \eqref{7.27} proves that 
\begin{align}
\fE_2(\tphith(t)) &=  (\nabla{\cdot}\pt\bsg(t)-\Pi_\cT^{k'}(\nabla{\cdot}\pt\bsg(t)),\Pi_\cT^{k'}(\tphith(t))-\tphith(t))_\Omega \nonumber\\&\quad- (\pt\bsg(t)-\bPi^k_{\cT}(\pt\bsg(t)),\nabla\tphith(t)-\bPi^k_\cT(\nabla\tphith(t)))_\Omega\nonumber\\&\quad + s_{\cM}((\Pi^{k'}_\cT(\pt v(t)),\Pi^k_\cF(\pt v(t)|_\cF)),(\Pi_\cT^{k'}(\tphith(t)),\Pi^k_{\cF}(\tphith(t)|_{\cF}))).
\end{align}
Invoking the Cauchy--Schwarz inequality, the approximation estimate \eqref{L2:proj-est} for the projection operator $\bPi^k_{\cT}$, and bounding the stabilization term as above leads to
\begin{align*}
|\fE_2(\tphith(t))| \lesssim{}&  h^s \Big\{ \norm{\pt\bsg(t)-\bPi^k_{\cT}(\pt\bsg(t))}_\Omega+\ell_\Omega\norm{\nabla (\pt v - {\cal E}_\cT^{k+1} (\pt v))}_\Omega \Big\} |\tphith(t)|_{H^{1+s}(\Omega)} \\
& + \norm{\nabla{\cdot}\pt\bsg(t)-\Pi_\cT^{k'}(\nabla{\cdot}\pt\bsg(t))}_\Omega \norm{\Pi_\cT^{k'}(\tphith(t))-\tphith(t)}_\Omega.
\end{align*}
The bound on the last term on the right-hand side depends on whether $k'=0$ or $k'\ge1$. 
Recalling that $\delta:=s$ if $k'=0$ and $\delta:=0$ otherwise, we obtain
\[
\norm{\Pi_\cT^{k'}(\tphith(t))-\tphith(t)}_\Omega \lesssim h^{1+s-\delta}
|\tphith(t)|_{H^{1+s-\delta}(\Omega)}.
\]
Altogether, this gives
\begin{align*}
|\fE_2(\tphith(t))| \lesssim{}&  h^s \Big\{ 
\big\{\norm{\pt\bsg(t)-\bPi^k_{\cT}(\pt\bsg(t))}_\Omega+\ell_\Omega\norm{\nabla (\pt v - {\cal E}_\cT^{k+1} (\pt v))}_\Omega\big\}|\tphith(t)|_{H^{1+s}(\Omega)} \\
& + h^{1-\delta}\ell_\Omega^{\delta} \norm{\nabla{\cdot}\pt\bsg(t)-\Pi_\cT^{k'}(\nabla{\cdot}\pt\bsg(t))}_\Omega |\tphith(t)|_{H^{1+s-\delta}(\Omega)}\Big\}.
\end{align*}
For $\fE_3$, the Cauchy--Schwarz inequality, the fractional multiplicative trace inequality (see \cite[Remark~12.19]{ErnGuermondI}), and the approximation property~\eqref{L2:proj-est} for $\bPi^k_{\cT}$ lead to 
\begin{align}
|\fE_3(\tphith(t))|&\leq \norm{\h{e}^{\HHO}_\cM(t)}_{\HHO}\Big\{\sum_{T\in\cT}h_T \norm{\nabla\tphith(t)-\bPi^{k}_T(\nabla\tphith(t))){\cdot}\bn_T}^2_{\dt}\Big\}^{\frac{1}{2}}\nonumber\\
&\lesssim h^s \norm{\h{e}^{\HHO}_\cM(t)}_{\HHO}|\tphith(t)|_{H^{1+s}(\Omega)}.\label{6.30}
\end{align}
Notice from \eqref{hho-error:1a} that $\bG_{\cT}(\h{e}^{\HHO}_\cM(t)) = \pt \bet_\cT^{\HHO}(t)$ for all $t\in\ol{J}$. This in 
Lemma~\ref{prop:stability} implies that
\begin{align}
\norm{\h{e}^{\HHO}_\cM(t)}_{\HHO} \lesssim \norm{\pt \bet_\cT^{\HHO}(t)}_\Omega+\ell_\Omega^{-\frac12}|\h{e}^{\HHO}_\cM(t)|_{\sS}.
\end{align}	
This in \eqref{6.30} proves that
\begin{align}
|\fE_3(\tphith(t))|\lesssim  h^s \Big\{ \norm{\pt \bet_\cT^{\HHO}(t)}_\Omega+\ell_\Omega^{-\frac12}|\h{e}^{\HHO}_\cM(t)|_{\sS} \Big\} |\tphith(t)|_{H^{1+s}(\Omega)}.\label{E6}
\end{align}
{\bf Step 4 (Conclusion)}. Collecting the previous estimates, invoking the Cauchy--Schwarz inequality in time, and using that $|\tphith|_{L^2(J;H^{1+s}(\Omega))} 
\le T_f^{\frac{1}{2}} |\tphith|_{L^\infty(J;H^{1+s}(\Omega))}$ proves that
\begin{align*}
|\fE(\Phith)| &\lesssim \norm{\Pi_\cT^{k'}(v_0)-\Pi^v_\cT(\ul{v}_0)}_{\Omega} 
\|\Phith\|_{L^\infty(J;L^2(\Omega))} \\
&+ h^s T_f^{\frac{1}{2}} \Big\{\big\{\ell_\Omega^{-\frac12} \abs{\h{e}_{\cM}^{\sPi}}_{L^2(J;\sS)} + \ell_\Omega^{\frac12} \abs{\pt\h{e}_{\cM}^{\sPi}}_{L^2(J;\sS)} + \norm{\pt\bet_\cT^{\sPi}}_{L^2(J;\bL^2(\Omega))} \\
&+ \ell_\Omega \norm{\nabla (\pt v - {\cal E}_\cT^{k+1} (\pt v))}_{L^2(J;\bL^2(\Omega))}
+ \norm{\pt\bsg-\bPi^k_{\cT}(\pt\bsg)}_{L^2(J;\bL^2(\Omega))} \big\}|\tphith|_{L^\infty(J;H^{1+s-\delta}(\Omega))}\\
&+h^{1-\delta}\ell_\Omega^\delta \norm{\nabla{\cdot}\pt\bsg-\Pi_\cT^{k'}(\nabla{\cdot}\pt\bsg)}_{L^2(J;L^2(\Omega))} |\tphith|_{L^\infty(J;H^{1+s}(\Omega))}\Big\} .
\end{align*}
Since $\|\Phith\|_{L^\infty(J;L^2(\Omega))}\lesssim \ell_\Omega \|\nabla \Phith\|_{L^\infty(J;\bL^2(\Omega))}$ owing to the Poincar\'e inequality in $\Omega$,
applying the regularity estimates \eqref{phi-reg} and \eqref{tphi-reg} and choosing $\theta := \ut{e}_\cT^{\sPi}(T_f)$ concludes the proof.

\bibliographystyle{siam}
\bibliography{references}
\end{document}